\documentclass[12pt]{amsart}
\usepackage{amssymb}
\usepackage{mathrsfs}
\usepackage{bm}
\usepackage{graphicx}
\usepackage{color}
\usepackage[all]{xy}

\allowdisplaybreaks[4]

\newtheorem{thm}{Theorem}[section]
\newtheorem{lem}[thm]{Lemma}
\newtheorem{cor}[thm]{Corollary}
\newtheorem{prop}[thm]{Proposition}
\newtheorem{definition}[thm]{Definition}

\def \para{\refstepcounter{thm} \par\medskip\noindent
                \textbf{\thethm .} }

\def \remark{\refstepcounter{thm} \par\medskip\noindent
                \textbf{Remark \thethm .} }

\numberwithin{equation}{thm}

\newcommand\BB{\mathbb B}
\newcommand\CC{\mathbb C}

\newcommand\ZZ{\mathbb Z}

\newcommand\bQ{\mathbf Q}

\newcommand\bb{\mathbf b}

\newcommand\bh{\mathbf h}

\renewcommand\bm{\mathbf m} 
\newcommand\bn{\mathbf n}

\newcommand\bu{\mathbf u}

\newcommand\cE{\mathcal{E}}

\newcommand\cI{\mathcal{I}}
\newcommand\cJ{\mathcal{J}}

\newcommand\cL{\mathcal{L}}
\newcommand\cM{\mathcal{M}}
\newcommand\cN{\mathcal{N}}

\newcommand\cX{\mathcal{X}}

\newcommand\fI{\mathfrak I}

\newcommand\fg{\mathfrak g}
\newcommand\fh{\mathfrak h}

\newcommand\fn{\mathfrak n}

\renewcommand\a{\alpha}  
\renewcommand\b{\beta}   
\newcommand\g{\gamma}  
\renewcommand\d{\delta}  

\renewcommand\t{\theta}  

\newcommand\la{\lambda}

\renewcommand\xi{\xi}
\renewcommand\pi{\pi}

\newcommand\w{\omega}

\newcommand\ve{\varepsilon}

\newcommand\vf{\varphi}

\renewcommand\Xi{\Xi}
\renewcommand\Pi{\Pi}

\newcommand\vG{\varGamma}

\newcommand\Bb{\boldsymbol\beta}  
\newcommand\Bg{\boldsymbol\gamma}

\newcommand\Bvf{\boldsymbol\varphi}

\newcommand{\dis}{\displaystyle}

\newcommand\wh{\widehat}
\newcommand\wt{\widetilde}

\newcommand\ra{\rightarrow}

\newcommand\LRa{\Leftrightarrow}

\newcommand\lan{\langle}
\newcommand\ran{\rangle}

\newcommand\rad{\operatorname{rad}}

\newcommand\Fgl{\mathfrak{gl}}
\newcommand\Fsl{\mathfrak{sl}}

\newcommand\ev{\mathbf{{ev}}}

\newcommand{\isom}{\,\raise2pt\hbox{$\underrightarrow{\sim}$}\,}

\newcounter{ichi}
\newcommand{\roi}{\roman{ichi}}
\newcounter{ni}
\newcommand{\roii}{\roman{ni}}
\newcounter{san}
\newcommand{\roiii}{\roman{san}}
\newcounter{yon}
\newcommand{\roiv}{\roman{yon}}
\newcounter{go}
\newcommand{\rov}{\roman{go}}
\newcounter{roku}
\newcommand{\rovi}{\roman{roku}}
\newcounter{nana}
\newcounter{hachi}
\newcounter{kyu}
\newcommand{\mo}{\operatorname{monic}}

\begin{document}

\setlength{\baselineskip}{4.9mm}
\setlength{\abovedisplayskip}{4.5mm}
\setlength{\belowdisplayskip}{4.5mm}


\renewcommand{\theenumi}{\roman{enumi}}
\renewcommand{\labelenumi}{(\theenumi)}
\renewcommand{\thefootnote}{\fnsymbol{footnote}}
\renewcommand{\thefootnote}{\fnsymbol{footnote}}
\parindent=20pt


\setcounter{section}{-1}




\address{Department of Mathematics, Faculty of Science, Shinshu University, 
		Asahi 3-1-1, Matsumoto 390-8621, Japan}
		
\email{wada@math.shinshu-u.ac.jp}



\medskip
\begin{center}
{\large \textbf{ Finite dimensional simple modules \\ of deformed current Lie algebras }}  
\\
\vspace{1cm}
Kentaro Wada 
\\[1em]
\end{center}


\title{} 
\maketitle 

\markboth{Kentaro Wada}{ Finite dimensional simple modules of deformed current Lie algebras }



\begin{abstract}
The deformed current Lie algebra was introduced in \cite{W} 
to study the representation theory of cyclotomic $q$-Schur algebras at $q=1$. 
In this paper, we classify finite dimensional simple modules of deformed current Lie algebras. 
\end{abstract}



\tableofcontents



\section{Introduction} 
\para 
The deformed current Lie algebra $\fg_{\wh{\bQ}}(\bm)$ was introduced in \cite{W} 
to study the representation theory of cyclotomic $q$-Schur algebras at $q=1$. 
In this paper, 
we introduce the deformed current Lie algebra 
$\Fsl_m^{\lan \bQ \ran}[x]$ and $\Fgl_m^{\lan \bQ \ran}[x]$ 
over $\CC$ 
associated with the special linear Lie algebra $\Fsl_m$ and general linear Lie algebra $\Fgl_m$ respectively. 
$\Fsl_m^{\lan \bQ \ran}[x]$ (resp. $\Fgl_m^{\lan \bQ \ran}[x]$) 
is a deformation of the current Lie algebra $\Fsl_m[x] = \Fsl_m \otimes_{\CC} \CC[x]$ 
(resp. $\Fgl_m[x] = \Fgl_m \otimes_{\CC} \CC [x]$) 
with deformation parameters $\bQ=(Q_1, Q_2, \dots, Q_{m-1}) \in \CC^{m-1}$. 
Note that $\Fsl_m^{\lan \bQ \ran}[x]$ (resp. $\Fgl_m^{\lan \bQ \ran}[x]$) 
is coincide with 
$\Fsl_m[x]$ (resp. $\Fgl_m[x]$) 
if $Q_i =0$ for all $i=1,2,\dots, m-1$. 
The Lie algebra $\fg_{\wh{\bQ}}(\bm)$ introduced in \cite{W} 
is isomorphic to $\Fgl_m^{\lan \bQ \ran}[x]$ under a suitable choice of deformation parameters $\bQ$ 
(Lemma \ref{Lemma iso slmQx gQm}). 

\para 
The differences of the representation theory of $\Fsl_m^{\lan \bQ \ran}[x]$ 
from one of $\Fsl_m [x]$ 
appear in the following two points. 
The deformed current Lie algebra $\Fsl_m^{\lan \bQ \ran}[x]$ has a family of $1$-dimensional representations 
$\{ \cL^{\Bb} \mid \Bb \in \prod_{i=1}^{m-1} \BB^{\lan Q_i \ran}\}$, 
where 
\begin{align*}
\BB^{\lan Q_i \ran} = 
\begin{cases} 
\{ 0 \} & \text{ if } Q_i =0, 
\\
\CC & \text{ if }  Q_i \not=0, 
\end{cases}
\end{align*}
although the $1$-dimensional representation of $\Fsl_m[x]$ is only the trivial representation 
(Lemma \ref{Lemma 1 dim slmx}). 
(We remark that $\cL^{(0,\dots,0)}$ is the trivial representation of $\Fsl_m^{\lan \bQ \ran}[x]$.) 

The second difference appears in the evaluation modules.  
For each $\g \in \CC$, 
we can consider the evaluation homomorphism 
$\ev_{\g} : U(\Fsl_m^{\lan \bQ \ran}[x]) \ra U(\Fsl_m)$ 
which is a deformation of the evaluation homomorphism for $\Fsl_m[x]$ 
(see the paragraph \ref{para evaluation sl2} for the definition). 
Then we can consider the evaluation modules 
by regarding $U(\Fsl_m)$-modules as $U(\Fsl_m^{\lan \bQ \ran}[x])$-modules 
through the evaluation homomorphism $\ev_{\g}$. 
The evaluation homomorphism $\ev_{\g}$ is surjective if $\g \not= Q_i^{-1}$ for all $i=1, 2,\dots, m-1$ such that $Q_i\not=0$. 
However, $\ev_{\g}$ is not surjective if $\g =Q_i^{-1}$ for some $i=1,2,\dots, m-1$. 
Moreover, in general, 
the evaluation module of a simple $U(\Fsl_m)$-module at $\g \in \CC$ 
is not simple if $\g =Q_i^{-1}$ for some $i=1,2,\dots, m-1$ (see Remark \ref{Remark ev not simple}). 

\para 
It is a purpose of this paper to classify the finite dimensional simple modules of 
$\Fsl_m^{\lan \bQ \ran}[x]$ and $\Fgl_m^{\lan \bQ \ran}[x]$. 
A classification of the finite dimensional simple modules for the original current Lie algebra is well-known 
(e.g. \cite{C}, \cite{CP}). 
The classification for $\Fsl_m^{\lan \bQ \ran}[x]$ (resp. $\Fgl_m^{\lan \bQ \ran}[x]$) 
is an analogue of the original case. 

Since $\Fsl_m^{\lan \bQ \ran}[x]$ has the triangular decomposition (Proposition \ref{Prop basis slmQ[x] and glmQ[x]}), 
we can develop the usual highest weight theory (see \S \ref{section Rep slmx}). 
In particular, 
any finite dimensional simple $U(\Fsl_m^{\lan \bQ \ran}[x])$-module 
is isomorphic to a highest weight module $\cL(\bu)$ of highest weight $\bu \in \prod_{i=1}^{m-1} \prod_{t \geq 0} \CC$ 
(Proposition \ref{Prop simple slmx HW}). 
Then it is enough to determine the highest weights such that 
the corresponding simple highest weight modules are finite dimensional. 
We obtain a classification of such highest weights as follows. 
Let $\CC[x]_{\mo}$ be the set of monic polynomials over $\CC$ with the indeterminate variable $x$. 
For each $Q \in \CC$, 
put 
\begin{align*}
\CC[x]_{\mo}^{\lan Q \ran} 
= \begin{cases} 
	\CC[x]_{\mo} & \text{ if } Q=0, 
	\\ 
	\{ \vf \in \CC[x]_{\mo} \mid Q^{-1} \text{ is not a root of } \vf \} & \text{ if } Q\not=0. 
	\end{cases} 
\end{align*}
We define the map 
$\prod_{i=1}^{m-1} (\CC[x]_{\mo}^{\lan Q_i \ran} \times \BB^{\lan Q_i \ran}) \ra \prod_{i=1}^{m-1} \prod_{t \geq 0} \CC$, 
\begin{align*}
(\Bvf, \Bb)=((\vf_i, \b_i))_{1 \leq i \leq m-1} \mapsto \bu^{\lan \bQ \ran} (\Bvf, \Bb) = (\bu^{\lan \bQ \ran} (\Bvf, \Bb)_{i,t})_{1\leq i \leq m-1, t \geq 0}, 
\end{align*}
by 
\begin{align}
\label{Intro uQ}
\bu^{\lan \bQ \ran} (\Bvf, \Bb)_{i,t} = 
	\begin{cases} 
		\g_{i,1}^t + \g_{i,2}^t + \dots + \g_{i,n_i}^t & \text{ if } Q_i=0, 
		\\
		\g_{i,1}^t + \g_{i,2}^t + \dots + \g_{i,n_i}^t  + Q_i^{-t} \b_i & \text{ if } Q_i \not=0 
	\end{cases}
\end{align}
when $\vf_i =(x- \g_{i,1})(x- \g_{i,2}) \dots (x - \g_{i,n_i})$ ($1\leq i \leq m-1$). 
Then we have the following classification of finite dimensional simple $U(\Fsl_m^{\lan \bQ \ran}[x])$-modules 
(Theorem \ref{Thm simple slmQ}).  
\begin{description}
\item[Theorem] 
$\{ \cL(\bu^{\lan \bQ \ran}(\Bvf, \Bb)) 
	\mid (\Bvf, \Bb) \in \prod_{i=1}^{m-1} (\CC[x]_{\mo}^{\lan Q_i \ran} \times \BB^{\lan Q_i \ran}) \}$ 
gives a complete set of isomorphism classes of finite dimensional simple $U(\Fsl_m^{\lan \bQ \ran}[x])$-modules. 
\end{description} 
We remark that $\cL(\bu^{\lan \bQ \ran}(\Bvf, \Bb))$ is isomorphic to a subquotient of 
\begin{align*} 
\big( \bigotimes_{j=1}^{m-1} \bigotimes_{k=1}^{n_j} L(\w_j)^{\ev_{\g_j,k}} \big) \otimes \cL^{\Bb}, 
\end{align*}
where $\{\w_j \mid 1\leq j \leq m-1\}$ is the set of fundamental weights for $\Fsl_m$, 
$L(\w_j)$ ($1 \leq j \leq m-1$) is the simple highest weight $U(\Fsl_m)$-module of highest weight $\w_j$ 
and $L(\w_j)^{\ev_{\g_j,k}}$ is the evaluation module of $L(\w_j)$ at $\g_{j,k}$. 

We also see that any finite dimensional simple $U(\Fgl_m^{\lan \bQ \ran}[x])$-module is 
isomorphic to a highest weight module $\cL (\wt{\bu})$ of highest weight $\wt{\bu} \in \prod_{j=1}^m \prod_{t \geq 0} \CC$ 
(Proposition \ref{Prop h.w. simple glmQ}). 
Note that $\Fsl_m^{\lan \bQ \ran}[x]$ is a Lie subalgebra of $\Fgl_m^{\lan \bQ \ran}[x]$ 
(Proposition \ref{Prop basis slmQ[x] and glmQ[x]} (\roiii)). 
The difference of representations of $\Fgl_m^{\lan \bQ \ran}[x]$ from one of $\Fsl_m^{\lan \bQ \ran}[x]$ 
is given by the family of $1$-dimensional $U(\Fgl_m^{\lan \bQ \ran}[x])$-modules 
$\{ \wt{\cL}^{\bh} \mid \bh \in \prod_{t \geq 0} \CC\}$. 
We remark that 
$\wt{\cL}^{\bh}$ ($\bh \in \prod_{t \geq 0} \CC$) 
is isomorphic to the trivial representation 
$\cL^{(0,\dots,0)}$ as a $U(\Fsl_m^{\lan \bQ \ran}[x])$-module 
when we restrict the action. 
We obtain the classification of finite dimensional simple $U(\Fgl_m^{\lan \bQ \ran}[x])$-modules as follows. 
We define the map 
$\prod_{i=1}^{m-1} ( \CC[x]_{\mo}^{\lan Q_i \ran} \times \BB^{\lan Q_i \ran}) \times \prod_{t \geq 0}\CC 
\ra \prod_{j=1}^m \prod_{t \geq 0} \CC$, 
\begin{align*}
(\Bvf, \Bb, \bh) =((\vf_i, \b_i)_{1\leq i \leq m-1}, (h_t)_{t \geq 0}) 
\mapsto \wt{\bu}^{\lan \bQ \ran} (\Bvf, \Bb, \bh)= (\wt{\bu}^{\lan \bQ \ran} (\Bvf, \Bb, \bh)_{j,t})_{1\leq j \leq m, t \geq 0}
\end{align*}
by 
\begin{align*}
\wt{\bu}^{\lan \bQ \ran}(\Bvf, \Bb, \bh)_{j,t} 
= \begin{cases} 
	\dis \sum_{k=j}^{m-1} \bu^{\lan \bQ \ran} (\Bvf, \Bb)_{k,t} + h_t & \text{ if } 1 \leq j \leq m-1 \text{ and } t \geq 0, 
	\\
	h_t & \text{ if } j=m \text{ and } t \geq 0, 
	\end{cases}
\end{align*}
where $\bu^{\lan \bQ \ran} (\Bvf, \Bb)_{k,t}$ is determined by \eqref{Intro uQ}. 
Then we have the following classification of finite dimensional simple $U(\Fgl_m^{\lan \bQ \ran}[x])$-modules 
(Theorem \ref{Thm class simple glmQ}). 
\begin{description}
\item[Theorem] 
$\{\cL(\wt{\bu}^{\lan \bQ \ran}(\Bvf, \Bb, \bh))
\mid (\Bvf, \Bb, \bh) \in \prod_{i=1}^{m-1} 
(\CC[x]_{\mo}^{\lan Q_i \ran} \times \BB^{\lan Q_i \ran}) \times \prod_{t \geq 0} \CC \}$ 
gives a complete set of isomorphism classes of finite dimensional simple $U(\Fgl_m^{\lan \bQ \ran}[x])$-modules. 
\end{description}
We remark that 
$\cL(\wt{\bu}^{\lan \bQ \ran}(\Bvf, \Bb, \bh))$ is isomorphic to a subquotient of 
\begin{align*}
\big( \bigotimes_{j=1}^{m-1} \bigotimes_{k=1}^{n_j} L(\wt{\w}_j)^{\wt{\ev}_{\g_j,k}} \big) 
\otimes \wt{\cL}^{\Bb} \otimes \wt{\cL}^{\bh}. 
\end{align*} 
(See \S \ref{Section class simple glmQ} for definitions of 
$L(\wt{\w}_j)^{\wt{\ev}_{\g_j,k}}$, $\wt{\cL}^{\Bb}$ and $\wt{\cL}^{\bh}$.) 
We also remark that 
\begin{align*}
&\cL(\wt{\bu}^{\lan \bQ \ran} (\Bvf, \Bb, \bh)) \cong \cL (\bu^{\lan \bQ \ran} (\Bvf, \Bb)), 
\\ 
&
L(\wt{\w}_j)^{\wt{\ev}_{\g_j,k}} \cong L(\w_j)^{\ev_{\g_j,k}}, 
\quad 
\wt{\cL}^{\Bb} \cong \cL^{\Bb} 
\text{ and } 
\wt{\cL}^{\bh} \cong \cL^{(0,\dots,0)} 
\end{align*}
as $U(\Fsl_m^{\lan\bQ \ran}[x])$-modules when we restrict the action. 
\\

{\bf Acknowledgements:} 
This work was supported by JSPS KAKENHI Grant Number JP16K17565. 



\section{Deformed current Lie algebras $\Fsl_m^{\lan \bQ \ran} [x]$ and $\Fgl_m^{\lan \bQ \ran} [x]$} 
In this section, 
we give a definition of deformed current Lie algebras $\Fsl_m^{\lan \bQ \ran}[x]$ and $\Fgl_m^{\lan \bQ \ran}[x]$, 
and also give some basic facts. 
The definition of $\Fgl_m^{\lan \bQ \ran}[x]$ in this section 
is different from one of $\fg_{\wh{\bQ}}(\bm)$ given in \cite{W}. 
The relation between $\Fgl_m^{\lan \bQ \ran}[x]$ and $\fg_{\wh{\bQ}}(\bm)$ is given in Lemma \ref{Lemma iso slmQx gQm}. 

\begin{definition} 
Put $\bQ =(Q_1,Q_2, \dots, Q_{m-1}) \in \CC^{m-1}$. 
We define the Lie algebra $\Fsl_m^{\lan \bQ \ran}[x]$ over $\CC$ by the following generators and 
defining relations: 
\begin{description}
\item[Generators] 
	$\cX_{i,t}^{\pm}$, $\cJ_{i,t}$ ($1 \leq i \leq m-1$,  $t \geq 0$). 
\item[Relations] 
\begin{align*}
&\tag{L1}
	[\cJ_{i,s}, \cJ_{j,t}]=0, 
\\
&\tag{L2} 
	[\cJ_{j,s}, \cX_{i,t}^{\pm}] = \pm a_{ji} \cX_{i, s+t}^{\pm}, 
\\
& \tag{L3} 
	[\cX_{i,t}^+, \cX_{j,s}^-] = \d_{ij} ( \cJ_{i,s+t} - Q_i \cJ_{i,s+t+1} ), 
\\
& \tag{L4} 
	[\cX_{i,t}^{\pm}, \cX_{j,s}^{\pm}] =0  \quad \text{ if } j \not= i \pm 1, 
\\ 
& \tag{L5}
	[\cX_{i,t+1}^+, \cX_{i \pm 1,s }^+] = [\cX_{i,t}^+, \cX_{i \pm 1, s+1}^+], 
	\quad 
	 [\cX_{i,t+1}^-, \cX_{i \pm1 ,s}^-] = [ \cX_{i,t}^-, \cX_{i \pm 1, s+1}^-], 
\\
& \tag{L6} 
	[\cX_{i,s}^+, [\cX_{i,t}^+, \cX_{i \pm 1,u}^+]] = [ \cX_{i,s}^-, [\cX_{i,t}^-, \cX_{i \pm 1,u}^-]] =0, 
\end{align*}
where we put 
$a_{ji} = 
	\begin{cases} 
		2 & \text{ if } j=i, 
		\\
		-1 & \text{ if } j = i \pm 1, 
		\\
		0 & \text{ otherwise}.
	\end{cases}
$
\end{description}

We also define the Lie algebra 
$\Fgl_m^{\lan \bQ \ran} [x]$ over $\CC$ by the following generators and 
defining relations: 
\begin{description}
\item[Generators] 
	$\cX_{i,t}^{\pm}$ ($1 \leq i \leq m-1$,  $t \geq 0$), 
	$\cI_{j,t}$ ($1 \leq j \leq m$, $t \geq 0$).
\item[Relations] 
\begin{align*}
&\tag{L'1}
	[\cI_{i,s}, \cI_{j,t}]=0, 
\\
&\tag{L'2} 
	[\cI_{j,s}, \cX_{i,t}^{\pm}] = \pm a'_{ji} \cX_{i, s+t}^{\pm}, 
\\
& \tag{L'3} 
	[\cX_{i,t}^+, \cX_{j,s}^-] = \d_{ij} ( \cJ_{i,s+t} - Q_i \cJ_{i,s+t+1} ), 
	\text{where we put } \cJ_{i,t} = \cI_{i,t} - \cI_{i+1,t}, 
\end{align*}
together with the relations (L4)-(L6)  in the above. 
In the relation (L'2), 
we put 
$a'_{ji} = 
	\begin{cases} 
		1 & \text{ if } j=i, 
		\\
		-1 & \text{ if } j = i + 1, 
		\\
		0 & \text{ otherwise}.
	\end{cases}
$
\end{description}
\end{definition}


\para 
We call $\Fsl_m^{\lan \bQ \ran}[x]$ (resp. $\Fgl_m^{\lan \bQ \ran} [x]$) 
the deformed current Lie algebra associated with the special linear Lie algebra $\Fsl_m$ 
(resp. the general linear Lie algebra $\Fgl_m$). 
If $Q_i=0$ for all $i=1,2,\dots,m-1$, 
then $\Fsl_m^{\lan \bQ \ran}[x]$  (resp. $\Fgl_m^{\lan \bQ \ran}[x]$) 
coincides with the current Lie algebra $\Fsl_m[x] = \Fsl_m \otimes_{\CC} \CC[x]$ 
(resp. $\Fgl_m[x] = \Fgl_m \otimes_{\CC} \CC[x]$)
associated with $\Fsl_m$ (resp. $\Fgl_m$).  
We can also regard $\Fsl_m^{\lan \bQ \ran}[x]$ (resp. $\Fgl_m^{\lan \bQ \ran}[x]$) 
as a filtered deformation of $\Fsl_m [x]$ (resp. $\Fgl_m[x]$) 
in a similar way as in \cite[Proposition 2.13]{W}. 


\para 
For $1 \leq i \not=j \leq m$ and $t \geq 0$, we define an element 
$\cE_{i,j;t} \in \Fsl_m^{\lan \bQ \ran} [x]$ (resp. $\cE_{i,j;t} \in \Fgl_n^{\lan \bQ \ran}[x]$) by 
\begin{align*}
\cE_{i,j;t} = 
	\begin{cases} 
		[\cX_{i,0}^+, [\cX_{i+1,0}^+, \dots, [\cX_{j-2,0}^+, \cX_{j-1,t}^+] \dots ]] & \text{ if } j >i, 
		\\
		[\cX_{i-1,0}^-, [\cX_{i-2,0}^-, \dots, [\cX_{j+1,0}^-, \cX_{j,t}^-] \dots ]] & \text{ if } j < i. 
	\end{cases}
\end{align*}
In particular, we have $\cE_{i,i+1;t} = \cX_{i,t}^+$ and $ \cE_{i+1,i;t} = \cX_{i,t}^-$. 

Let $\bn^+$ and $\bn^-$ be the Lie subalgebra of $\Fsl_m^{\lan \bQ \ran}[x]$ (also of $\Fgl_m^{\lan \bQ \ran}[x]$) 
generated by 
\begin{align*}
\{ \cX_{i,t}^+ \,|\, 1 \leq i \leq m-1, \, t \geq 0 \} 
\text{ and }
\{ \cX_{i,t}^- \,|\, 1 \leq i \leq m-1, \, t \geq 0 \} 
\end{align*}
respectively. 
Let $\bn^0$ (resp. $\wt{\bn}^0$) be the Lie subalgebra of $\Fsl_m^{\lan \bQ \ran} [x]$ (resp. $\Fgl_m^{\lan \bQ \ran}[x]$) 
generated by 
\begin{align*} 
\{ \cJ_{i,t} \,|\, 1 \leq i \leq m-1, \, t \geq 0\} 
\quad 
(\text{resp. }  \{ \cI_{i,t} \,|\, 1 \leq i \leq m, \, t \geq 0\}). 
\end{align*}
By the relation (L1) (resp. (L'1)), 
we see that 
$\bn^0$ (resp. $\wt{\bn}^0$) is a commutative Lie subalgebra of $\Fsl_m^{\lan \bQ} [x]$ (resp $\Fgl_m^{\lan \bQ \ran}[x]$). 


\begin{prop}\ 
\label{Prop basis slmQ[x] and glmQ[x]}
\begin{enumerate} 
\item 
	$\{\cE_{i,j;t} \,|\, 1 \leq i \not=j \leq m, \, t \geq 0\} \cup \{\cJ_{i,t} \,|\, 1 \leq i \leq m-1, \, t \geq 0\}$ 
	gives a basis of $ \Fsl_m^{\lan \bQ \ran}[x]$. 

\item 
	$\{\cE_{i,j;t} \,|\, 1 \leq i \not=j \leq m, \, t \geq 0\} \cup \{\cI_{j,t} \,|\, 1 \leq j \leq m, \, t \geq 0\}$ 
	gives a basis of $ \Fgl_m^{\lan \bQ \ran}[x]$. 

\item 
	There exists an injective homomorphism of Lie algebras 
	\begin{align*}
		\Upsilon : \Fsl_m^{\lan \bQ \ran} [x] \ra \Fgl_m^{\lan \bQ \ran} [x] 
		\text{ such that } 
		\cX_{i,t}^{\pm} \mapsto \cX_{i,t}^{\pm}, 
		\text{ and }
		\cJ_{i,t} \mapsto \cI_{i,t} - \cI_{i+1,t}. 
	\end{align*}
\item 
	We have the triangular decomposition 
	\begin{align*}
	\Fsl_m^{\lan \bQ \ran}[x] = \bn^- \oplus \bn^0 \oplus \bn^+ 
	\text{ and } 
	\Fgl_m^{\lan \bQ \ran} [x] = \bn^- \oplus \wt{\bn}^0 \oplus \bn^- 
	\quad ( \text{as vector spaces}).
	\end{align*}
	In particular, 
	\begin{align*} 
	\{\cE_{i,j ; t} \,|\, 1 \leq i < j \leq m, \, t \geq 0 \}  
	\quad (\text{resp. } \{\cE_{i,j ; t} \,|\, 1 \leq j < i \leq m, \, t \geq 0 \}) 
	\end{align*} 
	gives a basis of $\bn^+$ (resp. $\bn^-$), 
	and 
	\begin{align*} 
	\{\cJ_{i,t} \,|\, 1 \leq i \leq m-1, \, t \geq 0\} 
	\quad  (\text{resp. } \{\cI_{j,t} \,|\, 1 \leq j \leq m, \, t \geq 0\})
	\end{align*}  
	gives a basis of $\bn^0$ (resp. $\wt{\bn}^0$). 
\end{enumerate}
\end{prop}

\begin{proof} 
(\roi) and (\roii) are proven in a similar way as in the proof of \cite[Proposition 2.6]{W}. 
By checking the defining relations, 
we see that $\Upsilon$ is well-defined. 
We also see that $\Upsilon$ is injective by investigating the basis given in (\roi) and (\roii) 
under the homomorphism $\Upsilon$. 
Then we have (\roiii). 
(\roiv) folloes from (\roi) and (\roii). 
\end{proof}


\para \textbf{Evaluation homomorphisms and evaluation modules.} 
\label{para evaluation sl2}
The general linear Lie algebra $\Fgl_m$ is a Lie algebra over $\CC$ 
generated by $e_i, f_i$ ($1 \leq i \leq m-1$) and $K_j$ ($1 \leq  j \leq m$) 
together with the following defining relations: 
\begin{align*}
&[K_i, K_j]=0, 
\quad 
[K_j, e_i] =a'_{ji} e_i, 
\quad 
[K_j, f_i] = - a'_{ji} f_i, 
\\
&[e_i, f_j] = \d_{ij} H_i, \text{ where } H_i =K_i- K_{i+1}, 
\\
& [e_i,e_j] = [f_i,f_j] =0 \text{ if } j \not= i \pm 1, 
\quad 
[e_i, [e_i, e_{i \pm 1}]] = [f_i, [f_i, f_{i\pm 1}]] =0. 
\end{align*}
The special linear Lie algebra $\Fsl_m$ is a Lie subalgebra of $\Fgl_m$ generated by 
$e_i, f_i, H_i$ ($1\leq i \leq m-1$). 

For each $\g \in \CC$, by checking the defining relations, 
we have the homomorphisms of algebras 
(evaluation homomorphism) 
\begin{align*}
&\ev_\g :  U(\Fsl_m^{\lan \bQ \ran} [x]) \ra U(\Fsl_m) 
	\text{ by } 
	\cX_{i,t}^+ \mapsto (1 - Q_i \g ) \g^t e_i, \, 
	\cX_{i,t}^- \mapsto \g^t f_i, \, 
	\cJ_{i,t} \mapsto \g^t H_i 
\end{align*}
and 
\begin{align*} 
\wt{\ev}_{\g} : U(\Fgl_m^{\lan \bQ \ran}[x]) \ra U(\Fgl_m) 
	\text{ by } 
	\cX_{i,t}^+ \mapsto (1 - Q_i \g ) \g^t e_i, \, 
	\cX_{i,t}^- \mapsto \g^t f_i, \, 
	\cI_{j,t} \mapsto \g^t K_j. 
\end{align*} 
Clearly, 
the homomorphism $\ev_{\g}$ (resp. $\wt{\ev}_\g$) is surjective if $\g \not= Q_i^{-1}$ 
for all $i=1, \dots, m-1$ such that $Q_i \not=0$. 

For a $U(\Fsl_m)$-module $M$  (resp. a $U(\Fgl_m)$-module $M$), 
we can regard $M$ as a $U(\Fsl_m^{\lan \bQ \ran}[x])$-module (resp. a $U(\Fgl_m^{\lan \bQ \ran}[x])$-module) 
through the evaluation homomorphism $\ev_{\g}$ (resp. $\wt{\ev}_{\g}$). 
We call it the evaluation module, and denote it by $M^{\ev_\g}$ (resp. $M^{\wt{\ev}_\g}$).


\para 
In the rest of this section, 
we give a relation with the Lie algebra $\fg_{\wh{\bQ}}(\bm)$ introduced in \cite[Definition 2.2]{W}. 

Let $\bm =(m_1,\dots, m_r)$ be an $r$-tuple of positive integers such that 
$\sum_{k=1}^r m_k =m$. 
Put 
$\vG(\bm) = \{ (i,k) \,|\, 1 \leq i \leq m_k, \, 1 \leq k \leq r \}$ 
and $\vG'(\bm) = \vG(\bm) \setminus \{(m_r,r) \}$. 
Then we have the bijective map 
\begin{align*}
\zeta : \vG(\bm) \ra \{ 1 ,2, \dots, m\} 
\text{ such that } (i,k) \mapsto \sum_{j=1}^{k-1} m_j +i. 
\end{align*}
For $(i,k) \in \vG(\bm)$ and $j \in \ZZ$ such that $1 \leq \zeta((i,k)) + j \leq m$, 
put $(i+j, k) = \zeta^{-1} (\zeta((i,k))+j)$. 
For $(i,k) \in \vG'(\bm)$ and $(j,l) \in \vG(\bm)$, 
put $a'_{(j,l)(i,k)} = a'_{\zeta((j,l)) \zeta ((i,k))}$. 
Take $\wh{\bQ} = (\wh{Q}_1,\dots, \wh{Q}_{r-1}) \in \CC^{r-1}$. 
Then the Lie algebra 
$\fg_{\wh{\bQ}}(\bm)$ in \cite[Definition 2.2]{W} 
is defined by the generators 
$\cX_{(i,k),t}^{\pm}$, $\cI_{(j,l),t}$ ($(i,k) \in \vG'(\bm)$, $(j,l) \in \vG(\bm)$, $t \geq 0$) 
together with the following defining relations: 
\begin{align*}
& [\cI_{(i,k),s}, \cI_{(j,l),t}] =0, 
\quad 
	[\cI_{(j,l),s}, \cX_{(i,k),t}^{\pm}] = \pm a'_{(j,l)(i,k)} \cX_{(i,k),s+t}^{\pm}, 
\\
& [\cX_{(i,k),t}^+, \cX_{(j,l),s}^-]
 = \d_{(i,k)(j,l)} 
 	\begin{cases} 
		\cJ_{(i,k),s+t} & \text{ if } i \not= m_k, 
		\\
		- \wh{Q}_k \cJ_{(m_k,k),s+t} + \cJ_{(m_k,k),s+t+1} & \text{ if } i=m_k, 
	\end{cases}
\\
& [\cX_{(i,k),t}^{\pm}, \cX_{(j,l),s}^{\pm}]=0 \quad \text{ if } (j,l) \not= (i\pm 1,k), 
\\
& [\cX_{(i,k),t+1}^+, \cX_{(i \pm 1,k),s}^+] = [\cX_{(i,k),t}^+, \cX_{(i\pm1,k), s+1}^+], 
	\quad 
	[\cX_{(i,k),t+1}^-, \cX_{(i \pm 1,k),s}^-] = [\cX_{(i,k),t}^-, \cX_{(i\pm1,k), s+1}^-], 
\\
&  [\cX_{(i,k),s}^+, [\cX_{(i,k),t}^+, \cX_{(i \pm 1,k),u}^+]]
	= [\cX_{(i,k),s}^-, [\cX_{(i,k),t}^-, \cX_{(i \pm 1,k),u}^-]]
	=0, 
\end{align*}
where we put $\cJ_{(i,k),t} = \cI_{(i,k),t} - \cI_{(i+1,k),t}$. 
Then we have the following isomorphism 
between $\Fgl_m^{\lan \bQ \ran}[x]$ and $\fg_{\wh{\bQ}}(\bm)$ 
under the suitable choice of the deformation parameters $\bQ$. 

\begin{lem} 
\label{Lemma iso slmQx gQm}
Assume that $\wh{Q}_i \not=0$ for all $i=1,2,\dots,r-1$. 
We take $\bQ =(Q_1,Q_2,\dots, Q_{m-1}) \in \CC^{m-1}$ as 
\begin{align*}
Q_i = 
	\begin{cases} 
		\wh{Q}_k^{-1} & \text{ if } \zeta^{-1}(i) = (m_k,k) \text{ for some } k=1,2,\dots, r-1, 
		\\
		0 & \text{ otherwise}. 
	\end{cases}
\end{align*}
Then we have the isomorphism of Lie algebras 
$\Phi : \Fgl_m^{\lan \bQ \ran}[x] \ra \fg_{\wh{\bQ}}(\bm) $ such taht 
\begin{align*}
&\cX_{i,t}^+ \mapsto 
		\begin{cases} 
			\cX_{\zeta^{-1}(i), t}^+ 
				& \text{ if } \zeta^{-1}(i) \not= (m_k,k) \text{ for all } k=1,\dots, r-1, 
			\\
			- \wh{Q}_k^{-1} \cX_{\zeta^{-1}(i),t}^+ 
				&\text{ if } \zeta^{-1}(i)=(m_k,k) \text{ for some } k=1,\dots, r-1, 
		\end{cases}
\\
& \cX_{i,t}^- \mapsto \cX_{\zeta^{-1}(i),t}^-, 
	\quad 
	\cI_{j,t} \mapsto \cI_{\zeta^{-1}(j),t}. 
\end{align*}
\end{lem}

\begin{proof} 
We see the well-definedness of $\Phi$ by checking the defining relations. 
The inverse homomorphism of $\Phi$ is given by 
\begin{align*}
\cX_{(i,k),t}^+  \mapsto 
	\begin{cases} 
		\cX_{\zeta((i,k)),t}^+ & \text{ if } i \not= m_k,
		\\
		- \wh{Q}_k \cX_{\zeta((i,k)),t}^+ & \text{ if } i=m_k, 
	\end{cases} 
\quad 
\cX_{(i,k),t}^- \mapsto \cX_{\zeta((i,k)),t}^-, 
\quad 
\cI_{(j,l),t} \mapsto \cI_{\zeta((j,l)),t}. 
\end{align*}
\end{proof}



\section{Representations of $\Fsl_m^{\lan \bQ \ran}[x]$ } 
\label{section Rep slmx}

In this section, we give some fundamental results 
for finite dimensional $U(\Fsl_m^{\lan \bQ \ran}[x])$-modules 
by using the standard argument.  

\para 
Put 
$\fh = \bigoplus_{i=1}^{m-1} \CC \cJ_{i,0} \subset \Fsl_m^{\lan \bQ \ran}[x] $,  
then 
$\fh$ is a commutative Lie subalgebra of $\Fsl_m^{\lan \bQ \ran}[x]$. 
(Note that, if $Q_i =0$ for all $i=1,\dots,m-1$, 
$\fh$ is a Cartan subalgebra of $\Fsl_m$.) 
Let $\fh^{\ast}$  be the dual space of $\fh$. 
For each $i=1,2,\dots, m-1$, 
we take $\a_i \in \fh^\ast$ as 
$\a_i(\cJ_{j,0})= a_{ji}$ for $j=1,\dots,m-1$. 
Put $Q^+ = \bigoplus_{i=1}^{m-1} \ZZ_{\geq 0} \a_{i} \subset \fh^\ast$.  
We define the partial order on $\fh^{\ast} $ 
by  $\la \geq \mu$ if $\la - \mu \in Q^+$ for $\la,\mu \in \fh$. 


\para 
For $U(\Fsl_m^{\lan \bQ \ran}[x])$-mdoule $M$, we consider the decomposition 
$M = \bigoplus_{\la \in \fh^{\ast}} \wt{M}_{\la}$, where 
$\wt{M}_{\la} = \{ x \in M \,|\, ( h - \la (h))^N \cdot x =0 \text{ for }  h \in \fh \text{ and } N \gg 0\}$, 
namely $M = \bigoplus_{\la \in \fh^{\ast}} \wt{M}_{\la}$ is the decomposition 
to the generalized simultaneous eigenspaces  for the action of $\fh$. 
By the relation (L2), we have 
\begin{align*}
\cX_{i,t}^{\pm} \cdot \wt{M}_{\la} \subset \wt{M}_{\la \pm \a_i} 
\quad ( 1 \leq i \leq m-1, \, t \geq 0).
\end{align*}
Thus, if $U(\Fsl_m^{\lan \bQ \ran}[x])$-module $M\not=0$ is finite dimensional, 
there exists $\la \in \fh^{\ast}$ such that 
$\wt{M}_\la \not=0$ and $\cX_{i,t}^+ \cdot \wt{M}_{\la}=0$ for all $i=1,2,\dots, m-1$ and $t \geq 0$. 
On the other hand, $\wt{M}_{\la}$ ($\la \in \fh^{\ast}$) is closed under the action of $\bn^0$ by the relation (L1). 
Thus, we can take a simultaneous eigenvector $v \in \wt{M}_{\la}$ for the action of $\bn^0$. 
Then we have the following lemma. 


\begin{lem}
\label{Lemma finite dim module}
For a finite dimensional $U(\Fsl_m^{\lan \bQ \ran}[x])$-module $M \not=0$, 
there exists $v_0 \in M$ ($v_0 \not=0$) satisfying the following conditions: 
\begin{enumerate} 
\item 
$\cX_{i,t}^+ \cdot v_0 =0$ for all $i=1,\dots,m-1$ and $t \geq 0$, 

\item 
$\cJ_{i,t} \cdot v_0 = u_{i,t} v_0$ ($u_{i,t} \in \CC$) for each $i=1,\dots, m-1$ and $t \geq 0$. 
\end{enumerate}
Moreover, if $M$ is simple, we have $M = U(\Fsl_m^{\lan \bQ \ran}[x]) \cdot v_0$. 
\end{lem}


\para \textbf{Highest weight modules.} 
For $U(\Fsl_m^{\lan \bQ \ran}[x])$-module $M$, 
we say that $M$ is a highest weight module if there exists $v_0 \in M$ satisfying the following conditions: 
\begin{enumerate} 
\item 
$M$ is generated by $v_0$ as a $U(\Fsl_m^{\lan \bQ \ran}[x])$-module. 

\item 
$\cX_{i,t}^+ \cdot v_0 =0$ for all $i =1,\dots,m-1$ and $t \geq 0$. 

\item 
$\cJ_{i,t} \cdot v_0 = u_{i,t} v_0$ ($u_{i,t} \in \CC$) for each $i=1,\dots, m-1$ and $t \geq 0$. 
\end{enumerate}
In this case, we say that $(u_{i,t})_{1 \leq i \leq m-1, t \geq 0} \in \prod_{i=1}^{m-1} \prod_{t \geq 0} \CC$ 
is the highest weight of $M$, and that $v_0$ is a highest weight vector of $M$. 

Let $M$ be a highest weight $U(\Fsl_m^{\lan \bQ \ran}[x])$-module with a highest weight 
$\bu=(u_{i,t})_{1 \leq i \leq m-1, t \geq 0} \in \prod_{i=1}^{m-1} \prod_{t \geq 0} \CC$ 
and a highest weight vector $v_0 \in M$. 
Thanks to the triangular decomposition (Proposition \ref{Prop basis slmQ[x] and glmQ[x]} (\roiv)) 
together with the above conditions, 
we have 
$M= U(\bn^-) \cdot v_0$. 
Let $\la_{\bu} \in \fh^\ast$ be as $\la_{\bu} (\cJ_{i,0}) = u_{i,0}$ for $i=1,\dots,m-1$. 
By $M=U(\bn^-) \cdot v_0$ and  the relation (L2), 
we have the weight space decomposition 
\begin{align}
\label{wt sp decom h.w.}
M= \bigoplus_{\mu \in \fh^{\ast} \atop \mu \leq \la_{\bu}} M_{\mu}, 
\text{ where } 
M_{\mu} = \{x \in M \,|\, h \cdot x = \mu (h) \cdot x \text{ for } h \in \fh \}, 
\end{align}
and we also have $\dim_{\CC} M_{\la_{\bu}} =1$. 


\para \textbf{Verma modules.} 
For $\bu=(u_{i,t}) \in \prod_{i=1}^{m-1} \prod_{t \geq 0} \CC$, 
let $\fI(\bu)$ be the left ideal of $U(\Fsl_m^{\lan \bQ \ran}[x])$ generated by 
$\cX_{i,t}^+$ ($1 \leq i \leq m-1$, $t \geq 0$) and $\cJ_{i,t} - u_{i,t}$ ($1 \leq i \leq m-1$, $t \geq 0$). 
We define the Verma module $\cM(\bu) = U(\Fsl_m^{\lan \bQ \ran}[x])/ \fI(\bu)$. 
Then $\cM(\bu)$ is a highest weight module of highest weight $\bu$, 
and any highest weight module of highest weight $\bu$ is realized as a quotient of the Verma module $\cM(\bu)$. 
By the weight space decomposition \eqref{wt sp decom h.w.}, 
we see that $\cM(\bu)$ has the unique maximal proper submodule $\rad \cM(\bu)$. 
Put $\cL(\bu) = \cM(\bu)/ \rad \cM(\bu)$, then we have the following proposition. 


\begin{prop}
\label{Prop simple slmx HW}
For $\bu= (u_{i,t}) \in \prod_{i=1}^{m-1} \prod_{t \geq 0} \CC$, 
a highest weight simple $U(\Fsl_m^{\lan \bQ \ran}[x])$-module of highest weight $\bu$ is isomorphic to $\cL(\bu)$. 
Moreover, any finite dimensional simple $U(\Fsl_m^{\lan \bQ \ran} [x])$-module is isomorphic to $\cL(\bu)$ 
for some $\bu= (u_{i,t}) \in \prod_{i=1}^{m-1} \prod_{t \geq 0} \CC$. 
\end{prop}

\begin{proof} 
By Lemma \ref{Lemma finite dim module}, 
a finite dimensional simple $U(\Fsl_m^{\lan \bQ \ran}[x])$-module is a highest weight module. 
Then we have the proposition by the above arguments. 
\end{proof}



\section{Representations of $\Fgl_m^{\lan \bQ \ran}[x]$ } 

For finite dimensional $U(\Fgl_m^{\lan \bQ \ran}[x])$-modules, 
we can develop a similar argument as in the case of $U(\Fsl_m^{\lan \bQ \ran}[x])$ 
discussed in the previous section. 
In this section, 
we give only some notation for $U(\Fgl_m^{\lan \bQ \ran}[x])$-modules. 

\para \textbf{Highest weight modules.} 
For $U(\Fgl_m^{\lan \bQ \ran}[x])$-module $M$, 
we say that $M$ is a highest weight module if there exists $v_0 \in M$ satisfying the following conditions: 
\begin{enumerate} 
\item 
$M$ is generated by $v_0$ as a $U(\Fgl_m^{\lan \bQ \ran}[x])$-module. 

\item 
$\cX_{i,t}^+ \cdot v_0 =0$ for all $i =1,\dots,m-1$ and $t \geq 0$. 

\item 
$\cI_{j,t} \cdot v_0 = \wt{u}_{j,t} v_0$ ($\wt{u}_{j,t} \in \CC$) for each $j=1,\dots, m$ and $t \geq 0$. 
\end{enumerate}
In this case, we say that $(\wt{u}_{j,t})_{1 \leq j \leq m, t \geq 0} \in \prod_{j=1}^{m} \prod_{t \geq 0} \CC$ 
is the highest weight of $M$, and that $v_0$ is a highest weight vector of $M$. 


\para \textbf{Verma modules.} 
For $\wt\bu=(\wt{u}_{j,t}) \in \prod_{j=1}^{m} \prod_{t \geq 0} \CC$, 
let $\fI(\wt\bu)$ be the left ideal of $U(\Fgl_m^{\lan \bQ \ran}[x])$ generated by 
$\cX_{i,t}^+$ ($1 \leq i \leq m-1$, $t \geq 0$) and $\cI_{j,t} - \wt{u}_{j,t}$ ($1 \leq j \leq m$, $t \geq 0$). 
We define the Verma module $ \cM(\wt\bu) = U(\Fgl_m^{\lan \bQ \ran}[x])/ \fI(\wt\bu)$. 
Then $\cM(\wt\bu)$ is a highest weight module of highest weight $\wt\bu$, 
and any highest weight module of highest weight $\wt{\bu}$ is realized as a quotient of the Verma module $\cM(\wt\bu)$. 
$\cM(\wt{\bu})$ has the unique maximal proper submodule $\rad \cM(\wt\bu)$. 
Put $\cL(\wt\bu) = \cM(\wt\bu)/ \rad \cM(\wt\bu)$, then we have the following proposition. 


\begin{prop}
\label{Prop h.w. simple glmQ}
For $\wt\bu= (\wt{u}_{j,t}) \in \prod_{j=1}^{m} \prod_{t \geq 0} \CC$, 
a highest weight simple $U(\Fgl_m^{\lan \bQ \ran}[x])$-module of highest weight $\wt\bu$ is isomorphic to $\cL(\wt\bu)$. 
Moreover, any finite dimensional simple $U(\Fgl_m^{\lan \bQ \ran} [x])$-module is isomorphic to $\cL(\wt\bu)$ 
for some $\wt\bu= (u_{j,t}) \in \prod_{j=1}^{m} \prod_{t \geq 0} \CC$. 
\end{prop}



\section{Rank $1$ case ; some relations in $U (\Fsl_2^{\lan Q \ran} [x])$} 

\para 
Take $Q \in \CC$, then $\Fsl_2^{\lan Q \ran}[x]$ is a Lie algebra over $\CC$ 
generated by $\cX_t^{\pm}$ and $\cJ_t$ ($t \in \ZZ_{\geq 0}$) together with the following defining relations: 
\begin{align*}
&\tag{L1} 
[\cJ_s, \cJ_t] =0, 
\\ & \tag{L2} 
[\cJ_s, \cX_t^{\pm}] = \pm 2 \cX_{s+t}^{\pm}, 
\\ & \tag{L3} 
[\cX_t^+, \cX_s^-] = \cJ_{s+t} - Q \cJ_{s+t+1}, 
\\ & \tag{L4}
[\cX_t^{\pm}, \cX_s^{\pm}] =0.
\end{align*}
(In the rank $1$ case, we omit the first index of the generators since it is trivial.) 
By checking the defining relations, 
we see that there exists the algebra anti-automorphism  $ \dag : U(\Fsl_2^{\lan Q \ran}[x]) \ra U(\Fsl_2^{\lan Q \ran}[x])$ 
such that 
\begin{align}
\label{Def dag}
\dag (\cX_t^+)=\cX_t^-, \quad \dag (\cX_t^-) = \cX_t^+, \quad \dag (\cJ_t) = \cJ_t.
\end{align}
Clearly, $\dag^2$ is the identity on $U(\Fsl_2^{\lan Q \ran}[x])$. 


\para 
For $t, b \in \ZZ_{\geq 0}$, we define an element $\cX_t^{+ (b)}$ (resp. $\cX_t^{-(b)}$) of $ U(\Fsl_2^{\lan Q\ran}[x])$ by 
\begin{align*}
\cX_t^{\pm (b)} = \frac{(\cX_t^{\pm})^b}{b !}. 
\end{align*}
For  convenience, 
we put 
$\cX_t^{\pm(b)}=0$ for $b \in \ZZ_{<0}$.

For $t,p,h \in \ZZ_{\geq 0}$, 
we define an element $\cX_t^{+ ((p);h)}$ (resp. $\cX_t^{- ((p);h)}$) of $U(\Fsl_2^{\lan Q \ran}[x])$ by 
\begin{align}
\label{Def cX ph}
\cX_t^{\pm((0);h)} =1, 
\quad 
\cX_t^{\pm ((p);h)} = \sum_{w=0}^p \begin{pmatrix} p \\ w \end{pmatrix} (-Q)^w \cX_{t+ph +w}^{\pm} 
	\,\, \text{ for } p >0.
\end{align}
Clearly, we have $\dag (\cX_t^{+((p);h)}) = \cX_t^{-((p);h)}$. 
For examples, we have 
\begin{align*}
& \cX_t^{\pm ((0); h)} = 1, 
	\quad 
	\cX_t^{\pm ((1); h)} = \cX_{t+h}^{\pm} + (-Q) \cX_{t+h+1}^{\pm}, 
	\\ &
	\cX_t^{\pm ((2); h)} = \cX_{t+ 2 h}^{\pm} + 2 (-Q) \cX_{t+ 2 h +1}^{\pm} + (-Q)^2 \cX_{t+ 2 h +2}^{\pm}, 
	\\ & 
	\cX_t^{\pm ((3); h)} = \cX_{t+ 3 h}^{\pm} + 3 (-Q) \cX_{t+ 3 h +1}^{\pm} + 3 (-Q)^2 \cX_{t + 3 h +2}^{\pm} 
						+ (-Q)^3 \cX_{t+ 3 h +3}^{\pm}. 
\end{align*}

For $s,p \in \ZZ_{\geq 0}$, 
we define an element $\cJ_s^{\lan p \ran} $ of $ U(\Fsl_2^{\lan Q \ran}[x]))$ inductively on $p$ by 
\begin{align}
\label{Def cJ p}
\cJ_s^{\lan 0 \ran}=1, 
\quad 
\cJ_s^{\lan p \ran} 
	= \frac{1}{p} \sum_{z=1}^p (-1)^{z-1} 
		\Big( \sum_{w=0}^z \begin{pmatrix} z \\ w \end{pmatrix}  (-Q)^w \cJ_{z s + w}\Big) \cJ_s^{\lan p-z \ran} 
	\text{ for } p >0. 
\end{align}
For examples, we have 
\begin{align*}
& \cJ_s^{\lan 0 \ran}=1, 
	\quad \cJ_s^{\lan 1 \ran} = \cJ_s + (-Q) \cJ_{s+1}, 
\\
& \cJ_s^{\lan 2 \ran} = 
	\frac{1}{2} \Big( \big( \cJ_s^2 - \cJ_{ 2 s} \big) + 2 (-Q) \big( \cJ_s \cJ_{s+1} - \cJ_{ 2 s+1} \big) 
		+ (-Q)^2 \big( \cJ_{s+1}^2  - \cJ_{ 2 s +2} \big) \Big), 
\\
& \cJ_s^{\lan 3 \ran} = 
	\frac{1}{3} \Big( \big( \cJ_s^3 - 2 \cJ_s \cJ_{2 s} + \cJ_{ 3 s} \big) 
		+ 3 (-Q) \big( \cJ_s^2 \cJ_{s+1} - \cJ_s \cJ_{ 2 s +1} - \cJ_{s+1} \cJ_{2 s} + \cJ_{3 s +1} \big) 
		\\ & \hspace{5em}
		+ (-Q)^2 \big( 3 \cJ_s \cJ_{s+1}^2 - 2 \cJ_s \cJ_{ 2 s +2} - 4 \cJ_{s+1} \cJ_{ 2 s +1} + 3 \cJ_{ 3 s +2} \big)  
		\\ & \hspace{5em}
		+ (-Q)^3 \big( \cJ_{s+1}^3 - 2 \cJ_{s+1} \cJ_{ 2 s + 2} + \cJ_{ 3 s +3} \big) \Big). 
\end{align*}


\begin{lem} 
\label{Lemma cJsp cXt}
For $s,t,p \in \ZZ_{\geq 0}$, we have the following relations in $U(\Fsl_2^{\lan Q \ran}[x])$. 
\begin{enumerate} 
\item 
	$\dis [\cJ_s^{\lan p \ran}, \cX_t^+] = \sum_{z=1}^p (-1)^{z+1} (z+1) \cJ_s^{\lan p-z \ran} \cX_t^{+((z);s)}$. 

\item 
	$\dis [ \cJ_s^{\lan p \ran}, \cX_t^-] = - \sum_{z=1}^p (-1)^{z+1} (z+1) \cX_t^{-((z);s)} \cJ_s^{\lan p-z \ran}$. 
\end{enumerate}
\end{lem}

\begin{proof} 
(\roii) follows from (\roi) by applying the algebra anti-automorphism $\dag$ defined in \eqref{Def dag}. 
Then, we prove only (\roi) by the induction on $p$. 

If $p=0$, (\roi) is clear. 
If $p >0$, by the definition \eqref{Def cJ p}, we have 
\begin{align*}
\cJ_s^{\lan p \ran} \cX_t^+ 
&= \frac{1}{p} \sum_{z=1}^p (-1)^{z-1} \Big( \sum_{w=0}^z \begin{pmatrix} z \\ w \end{pmatrix} (-Q)^w \cJ_{zs +w} \Big) 
	\cJ_s^{\lan p-z \ran} \cX_t^+. 
\end{align*}
By the assumption of the induction, we have 
\begin{align*}
\cJ_s^{\lan p \ran} \cX_t^+ 
&= \frac{1}{p} \sum_{z=1}^p (-1)^{z-1} \Big( \sum_{w=0}^z \begin{pmatrix} z \\ w \end{pmatrix} (-Q)^w \cJ_{zs +w} \Big) 
	\\ & \quad \times 
	\Big(  \cX_t^+ \cJ_s^{\lan p-z \ran} + \sum_{k=1}^{p-z} (-1)^{k+1} (k+1) \cJ_s^{\lan p-z-k \ran} \cX_t^{+((k);s)} \Big) 
\\
&= \frac{1}{p} \sum_{z=1}^p (-1)^{z-1} \sum_{w=0}^z \begin{pmatrix} z \\ w \end{pmatrix} (-Q)^w 
	(\cX_t^+ \cJ_{zs +w} + 2 \cX_{t + zs +w}^+) \cJ_s^{\lan p-z \ran} 
	\\ & \quad 
	+ \frac{1}{p} \sum_{z=1}^p \sum_{w=0}^z \sum_{k=1}^{p-z} 
		(-1)^{z+k} \begin{pmatrix} z \\ w \end{pmatrix} (k+1) (-Q)^w \cJ_{zs +w} \cJ_s^{\lan p-z-k \ran} \cX_t^{+((k);s)}. 
\end{align*}
Applying the assumption of the induction again, we have 
\begin{align}
\begin{split} 
\label{cJsp cXt}
\cJ_s^{\lan p \ran} \cX_t^+ 
&= \cX_t^+ \frac{1}{p} \sum_{z=1}^p (-1)^{z-1} \sum_{w=0}^z \begin{pmatrix} z \\ w \end{pmatrix} (-Q)^w 
		\cJ_{zs +w} \cJ_s^{\lan p-z \ran} 
	\\ &\quad 
	+ 2 \frac{1}{p} \sum_{z=1}^p (-1)^{z-1} \sum_{w=0}^z \begin{pmatrix} z \\ w \end{pmatrix} (-Q)^w 
		\cJ_s^{\lan p-z \ran} \cX_{t+zs +w}^+
	\\ & \quad 
	- 2 \frac{1}{p} \sum_{z=1}^p (-1)^{z-1} \sum_{w=0}^z \begin{pmatrix} z \\ w \end{pmatrix} (-Q)^w 
		\sum_{k=1}^{p-z} (-1)^{k+1} (k+1) \cJ_s^{\lan p-z-k \ran} \cX_{t+ zs +w}^{+((k);s)} 
	\\ & \quad 
	+ \frac{1}{p} \sum_{z=1}^p \sum_{w=0}^z \sum_{k=1}^{p-z} 
		(-1)^{z+k} \begin{pmatrix} z \\ w \end{pmatrix} (k+1) (-Q)^w \cJ_{zs +w} \cJ_s^{\lan p-z-k \ran} \cX_t^{+((k);s)}. 
\end{split}
\end{align}
By the definition \eqref{Def cJ p}, we have 
\begin{align}
\label{cXt sum = cXt cJsp}
\cX_t^+ \frac{1}{p} \sum_{z=1}^p (-1)^{z-1} \sum_{w=0}^z \begin{pmatrix} z \\ w \end{pmatrix} (-Q)^w 
		\cJ_{zs +w} \cJ_s^{\lan p-z \ran} 
= \cX_t^+ \cJ_s^{\lan p \ran}. 
\end{align}
By the definition \eqref{Def cX ph}, we have 
\begin{align}
\label{sum cJs cXt = sum cJs cXt}
\sum_{z=1}^p (-1)^{z-1} \sum_{w=0}^z \begin{pmatrix} z \\ w \end{pmatrix} (-Q)^w \cJ_s^{\lan p-z \ran} \cX_{t+zs +w}^+
= \sum_{z=1}^p (-1)^{z-1} \cJ_s^{\lan p-z \ran} \cX_t^{+((z);s)}.
\end{align}
Put 
\begin{align*}
(\ast)
&=\sum_{z=1}^p (-1)^{z-1} \sum_{w=0}^z \begin{pmatrix} z \\ w \end{pmatrix} (-Q)^w 
		\sum_{k=1}^{p-z} (-1)^{k+1} (k+1) \cJ_s^{\lan p-z-k \ran} \cX_{t+ zs +w}^{+((k);s)}. 
\end{align*}
By the definition \eqref{Def cX ph}, we also have 
\begin{align*}
(\ast)
&= \sum_{z=1}^p (-1)^{z-1} \sum_{w=0}^z \begin{pmatrix} z \\ w \end{pmatrix} (-Q)^w 
		\sum_{k=1}^{p-z} (-1)^{k+1} (k+1) \cJ_s^{\lan p-z-k \ran} 
		\\ & \quad \times 
		\Big( \sum_{l=0}^k \begin{pmatrix} k \\ l \end{pmatrix} (-Q)^l \cX_{(t+zs +w) + k s +l}^+ \Big) 
\\
&= \sum_{z=1}^p \sum_{k=1}^{p-z} (-1)^{z+k} (k+1) \cJ_s^{\lan p - (z+k) \ran} 
	\sum_{w=0}^z \sum_{l=0}^k \begin{pmatrix} z \\ w \end{pmatrix} \begin{pmatrix} k \\ l \end{pmatrix} (-Q)^{w+l} 
		\cX_{t + (z+k)s + (w+l)}^+ 
\end{align*}
Put $z'=z+k$ and $w'=w+l$, we have 
\begin{align*}
(\ast) 
&= \sum_{z'=2}^p \sum_{k=1}^{z'-1} (-1)^{z'} (k+1) \cJ_s^{\lan p - z' \ran} 
	\sum_{w'=0}^{z'} \sum_{ l= \max \{ 0, w'-(z'-k) \}}^{\min\{k, w' \}} 
	\begin{pmatrix} z' - k \\ w' -l \end{pmatrix} \begin{pmatrix} k \\ l \end{pmatrix} (-Q)^{w'} \cX_{t+z' s + w'}^+.
\end{align*}
By the induction on $k$, 
we can show that 
\begin{align}
\label{sum max min} 
\sum_{ l= \max \{ 0, w'-(z'-k) \}}^{\min\{k, w' \}} 
	\begin{pmatrix} z' - k \\ w' -l \end{pmatrix} \begin{pmatrix} k \\ l \end{pmatrix} 
= \begin{pmatrix} z' \\ w' \end{pmatrix}.
\end{align}
Then, we have 
\begin{align*}
(\ast) = \sum_{z'=2}^p  (-1)^{z'} \Big( \sum_{k=1}^{z'-1} (k+1) \Big) \cJ_s^{\lan p - z' \ran} 
	\sum_{w'=0}^{z'} \begin{pmatrix} z' \\ w' \end{pmatrix} (-Q)^{w'} \cX_{t+z' s + w'}^+, 
\end{align*}
and by the definition of \eqref{Def cX ph}, we have 
\begin{align}
\label{ast = sum cJs cXt}
(\ast) = \sum_{z'=2}^p (-1)^{z'} \frac{(z'-1)(z'+2)}{2} \cJ_s^{\lan p-z' \ran} \cX_t^{+((z');s)}.
\end{align}
By the definition \eqref{Def cJ p}, we have 
\begin{align}
\label{sumsumsum = sum cJs cXt}
\begin{split}
&\sum_{z=1}^p \sum_{w=0}^z \sum_{k=1}^{p-z} 
		(-1)^{z+k} \begin{pmatrix} z \\ w \end{pmatrix} (k+1) (-Q)^w \cJ_{zs +w} \cJ_s^{\lan p-z-k \ran} \cX_t^{+((k);s)}
\\
&= \sum_{k=1}^{p-1} \sum_{z=1}^{p-k} \sum_{w=0}^z 
	(-1)^{z+k} \begin{pmatrix} z \\ w \end{pmatrix} (k+1) (-Q)^w \cJ_{zs +w} \cJ_s^{\lan p-z-k \ran} \cX_t^{+((k);s)}
\\
&= \sum_{k=1}^{p-1} (-1)^{k+1} (k+1) (p-k) 
	\Big( \frac{1}{p-k} \sum_{z=1}^{p-k} (-1)^{z-1} \sum_{w=0}^z \begin{pmatrix} z \\ w \end{pmatrix} 
		(-Q)^w \cJ_{zs + w} \cJ_s^{\lan (p-k) -z \ran} \Big) \cX_t^{+((k);s)}
\\
&= \sum_{k=1}^{p-1} (-1)^{k+1} (k+1)(p-k) \cJ_s^{\lan p - k \ran} \cX_t^{+((k);s)}.
\end{split}
\end{align}
Combining \eqref{cJsp cXt} with \eqref{cXt sum = cXt cJsp}, \eqref{sum cJs cXt = sum cJs cXt}, \eqref{ast = sum cJs cXt} 
and \eqref{sumsumsum = sum cJs cXt}, we have 
\begin{align*}
\cJ_s^{\lan p \ran} \cX_t^+ 
&= \cX_t^+ \cJ_s^{\lan p \ran} 
	+ \frac{1}{p} \sum_{z=1}^p (-1)^{z+1} \big( 2 + (z-1)(z+2) + (z+1)(p-z)\big)\cJ_s^{\lan p-z \ran} \cX_t^{+((z);s)}
\\
&= \cX_t^+ \cJ_s^{\lan p \ran}  
	+ \sum_{z=1}^p (-1)^{z+1} (z+1) \cJ_s^{\lan p-z \ran} \cX_t^{+((z);s)}. 
\end{align*}
\end{proof}


\begin{lem} 
\label{Lemma cXt+ cXs-(p)h}
For $s,t,h \in \ZZ_{\geq 0}$ and $p \in \ZZ_{>0}$, we have 
\begin{align*}
[\cX_t^+, \cX_s^{-((p);h)}] = \sum_{w=0}^p \begin{pmatrix} p \\ w \end{pmatrix} (-Q)^w \cJ_{s+t+ p h +w}^{\lan 1 \ran}. 
\end{align*}
\end{lem}

\begin{proof} 
By the definitions \eqref{Def cX ph}, \eqref{Def cJ p} and the defining relation (L3), 
we have 
\begin{align*}
\cX_t^+ \cX_s^{-((p);h)} 
&= \sum_{w=0}^p \begin{pmatrix} p \\ w \end{pmatrix} (-Q)^w \cX_t^+ \cX_{s+ p h +w}^{-} 
\\
&= \sum_{w=0}^p \begin{pmatrix} p \\ w \end{pmatrix} (-Q)^w 
	\big( \cX_{s + p h +w}^- \cX_t^+ + \cJ_{s+t+ph+w} + (- Q) \cJ_{s+t+ph+w+1} \big)
\\
&= \cX_s^{-((p);h)} \cX_t^+ + \sum_{w=0}^p \begin{pmatrix} p \\ w \end{pmatrix} (-Q)^w \cJ_{s+t+ph +w}^{\lan 1 \ran}. 
\end{align*} 
\end{proof}


\begin{lem} 
\label{Lemma cJs1 cXtph}
For $s,t,h \in \ZZ_{\geq 0}$ and $p \in \ZZ_{>0}$, 
we have the following relations. 
\begin{enumerate} 
\item 
	$\dis [\cJ_s^{\lan 1 \ran}, \cX_t^{+((p);h)}] = 2 \cX_{s+t-h}^{+((p+1);h)}$. 

\item 
	$\dis [\cJ_s^{\lan 1 \ran}, \cX_t^{-((p);h)}] = - 2 \cX_{s+t-h}^{-((p+1);h)}$. 

\end{enumerate} 
\end{lem}

\begin{proof} 
(\roii) follows from (\roi) by applying the algebra anti-automorphism $\dag$ defined in \eqref{Def dag}. 
Then, we prove (\roi). 

By the definition \eqref{Def cX ph}, we have 
\begin{align*}
\cJ_s^{\lan 1 \ran} \cX_t^{+((p);h)} 
= \sum_{w=0}^p \begin{pmatrix} p \\ w \end{pmatrix} (-Q)^w  \cJ_s^{\lan 1 \ran} \cX_{t+ p h +w}^+. 
\end{align*}
Applying Lemma \ref{Lemma cJsp cXt} (\roi), we have 
\begin{align*}
\cJ_s^{\lan 1 \ran} \cX_t^{+((p);h)} 
&= \sum_{w=0}^p \begin{pmatrix} p \\ w \end{pmatrix} (-Q)^w 
	 \big( \cX_{t+ph +w}^+ \cJ_s^{\lan 1 \ran} + 2  \cX_{t+ p h +w}^{+((1);s)} \big) 
\end{align*}
Then, by the definition \eqref{Def cX ph} again, we have 
\begin{align*}
\cJ_s^{\lan 1 \ran} \cX_t^{+((p);h)}
&=\cX_t^{+((p);h)} \cJ_s^{\lan 1\ran} + 2 \sum_{w=0}^p \begin{pmatrix} p \\ w \end{pmatrix} (-Q)^w 
	\big( \cX_{s+t+ph+w}^+ +(-Q) \cX_{s+t+ph+w+1}^+ \big). 
\end{align*}
On the other hand, we have 
\begin{align*}
&\sum_{w=0}^p \begin{pmatrix} p \\ w \end{pmatrix} (-Q)^w \big( \cX_{s+t+ph+w}^+ +(-Q) \cX_{s+t+ph+w+1}^+ \big) 
\\
&= \cX_{s+t+ph}^+ 
	+ \sum_{w=1}^p \big\{ \begin{pmatrix} p \\ w \end{pmatrix} + \begin{pmatrix} p \\ w-1 \end{pmatrix} \big\} 
		(-Q)^w \cX_{s+t+ph+w}^+ 
	+ (-Q)^{p+1} \cX_{s+t+ph+p+1}^+ 
\\
&= \sum_{w=0}^{p+1} \begin{pmatrix} p+1 \\ w \end{pmatrix} (-Q)^w \cX_{s+t-h +(p+1) h+w}^+ 
\\
&= \cX_{s+t-h}^{+(p+1);h}. 
\end{align*}
Thus, we have (\roi). 
\end{proof}

\begin{lem} 
\label{Lemma cXt cXsc}
For $s,t,c \in \ZZ_{\geq 0}$, we have 
\begin{align*} 
[\cX_t^+, \cX_s^{-(c)}] = \cX_s^{-(c-1)} \cJ_{s+t}^{\lan 1 \ran} - \cX_s^{-(c-2)} \cX_s^{-((1);s+t)}. 
\end{align*} 
\end{lem}

\begin{proof} 
We prove the lemma by the induction on $c$. 
If $c=0$, it is clear. 
If $c=1$, it is the defining relation (L3). 
If $c >1$, 
by the assumption of the induction, 
we have 
\begin{align*}
\cX_t^+ \cX_s^{-(c)} 
&= \frac{1}{c} \cX_t^+ \cX_s^{-(c-1)} \cX_s^- 
\\
&= \frac{1}{c} 
	\big( \cX_s^{-(c-1)} \cX_t^+ + \cX_s^{-(c-2)} \cJ_{s+t}^{\lan 1 \ran} - \cX_{s}^{-(c-3)} \cX_s^{-((1);s+t)} \big) 
	\cX_s^-. 
\end{align*}
Then, by the defining relations (L3), (L4) and Lemma \ref{Lemma cJsp cXt} (\roii), we have 
\begin{align*}
\cX_t^+ \cX_s^{-(c)} 
&= \frac{1}{c} \big\{ 
	\cX_s^{-(c-1)} \big( \cX_s^- \cX_t^+ + \cJ_{s+t}^{\lan 1 \ran} \big)
	+ \cX_s^{-(c-2)} \big( \cX_s^- \cJ_{s+t}^{\lan 1 \ran} - 2 \cX_{s}^{-((1); s+t)}\big) 
	\\ & \hspace{3em} 
	- \cX_s^{-(c-3)} \cX_s^- \cX_s^{-((1);s+t)} \big\}  
\\
&= \frac{1}{c} \big\{ c \cX_s^{-(c)} \cX_t^+ + \cX_s^{-(c-1)} \cJ_{s+t}^{\lan 1 \ran} + (c-1) \cX_s^{-(c-1)} \cJ_{s+t}^{\lan 1 \ran}
	\\ & \hspace{2em}  
	- 2 \cX_s^{-(c-2)} \cX_s^{-((1);s+t)} - (c-2) \cX_s^{-(c-2)} \cX_s^{-((1);s+t)} \big\} 
\\
&= \cX_s^{-(c)} \cX_t^+ + \cX_s^{-(c-1)} \cJ_{s+t}^{\lan 1 \ran} - \cX_s^{-(c-2)} \cX_s^{-((1);s+t)}. 
\end{align*}
\end{proof}

\para 
A partition $\la=(\la_1, \la_2, \dots)$ is a non-increasing sequence of non-negative integers 
which has only finitely many non-zero terms. 
The size of a partition $\la$ is the sum of all terms of $\la$, and we denote it by $|\la|$. 
Namely, we have $|\la|=\sum_{i \geq 1} \la_i$. 
If $|\la|=n$, we say that $\la$ is a partition of $n$, and we denote it  by $\la \vdash n$. 
The length of $\la$ is the maximal $i$ such that $\la_i \not=0$, 
and we denote the length of $\la$ by $\ell(\la)$. 
For a partition $\la=(\la_1,\la_2,\dots)$, 
let $m_j (\la)$ ($j \in \ZZ_{>0}$) be the multiplicity of $j $ in $\la$. 
Then, for a partition $\la$ and $t,h \in \ZZ_{\geq 0}$, 
we define an element $\cX_t^{+(\la;h)}$ (resp. $\cX_t^{-(\la;h)}$) of $ U(\Fsl_2^{\lan Q \ran}[x])$ by 
\begin{align}
\label{Def cXlah}
\cX_t^{\pm (\la;h)} = \prod_{j \geq 1} \frac{(\cX_t^{\pm ((j); h)})^{m_j(\la)}}{m_j(\la) !}, 
\end{align}
where we note the defining relation (L4).
Clearly, we have $\dag(\cX_t^{+(\la; h)}) = \cX_t^{-(\la;h)}$. 
For examples, we have 
\begin{align*}
& \cX_t^{\pm ((0);h)} =1, 
	\quad 
	\cX_t^{\pm ((1); h)} = \cX_t^{\pm ((1); h)}, 
\\
& \cX_t^{\pm ((2);h)} = \cX_t^{\pm((2);h)}, 
	\quad 
	\cX_t^{\pm((1,1);h)} = \frac{(\cX_t^{\pm ((1);h)})^2}{2 !}, 
\\
&  \cX_t^{\pm ((3);h)} = \cX_t^{\pm ((3);h)}, 
	\quad 
	\cX_t^{\pm ((2,1);h)} = \cX_t^{\pm ((2);h)} \cX_t^{\pm ((1);h)}, 
	\quad 
	\cX_t^{\pm ((1,1,1);h)} = \frac{(\cX_t^{\pm (1);h)})^3}{3!}, 
\\
& \cX_t^{\pm ((3,3,2,2,2,1,1);h)} = \frac{(\cX_t^{\pm ( (3) ;h)})^2}{2 !} \frac{(\cX_t^{\pm ((2);h)})^3}{ 3 !} 
	\frac{(\cX_t^{+((1);h)})^2}{2 !}. 
\end{align*}

For $t,h,k,b,p \in \ZZ_{\geq 0}$, 
we define an element 
$\cX_t^{+(b;p | k;h)}$ (resp. $\cX_t^{-(b;p | k;h)}$)  of $U(\Fsl_2^{\lan Q \ran}[x])$ by 
\begin{align}
\label{Def cXbpkh}
\cX_t^{\pm (b;p|k;h)} = \sum_{\la \vdash k} \cX_t^{\pm (\la;h)} \cX_t^{\pm (b-p - \ell(\la))}.
\end{align}
Note the defining relation (L4), 
we see that $\dag (\cX_t^{+(b;p|k;h)}) = \cX_t^{-(b;p | k;h)}$. 
For examples, we have 
\begin{align*}
& \cX_t^{\pm (b;p|0;h)} = \cX_t^{\pm (b-p)}, 
	\quad 
	\cX_t^{\pm (b;p|1;h)} = \cX_t^{\pm ((1);h)} \cX_t^{\pm (b-p-1)}, 
\\ & 
	\cX_t^{\pm (b;p|2;h)} = \cX_t^{\pm ((2);h)} \cX_t^{\pm (b-p-1)} + \cX_t^{\pm ((1,1);h)} \cX_t^{\pm (b-p-2)}, 
\\ & 
	\cX_t^{\pm (b;p | 3 ;h)} = \cX_t^{\pm((3);h)} \cX_t^{\pm (b-p-1)} + \cX_t^{\pm ((2,1);h)} \cX_t^{\pm (b-p-2)} 
		+ \cX_t^{\pm ((1,1,1);h)} \cX_t^{\pm (b-p-3)}.  
\end{align*}

For the element $\cX_t^{\pm (b;p|k;h)} \in U(\Fsl_2^{\lan Q \ran}[x])$, 
we prepare the following technical formulas. 

\begin{lem}
\label{Lemma cXtbpkh}
For $t,h,k,b,p \in \ZZ_{\geq 0}$, we have the following equations for the element 
$\cX_t^{\pm (b;p|k;h)}$ of $U(\Fsl_2^{\lan Q \ran}[x])$. 
\begin{enumerate} 
\item 
If $b-p <0$, we have $\cX_t^{\pm (b;p|k;h)} =0$. 

\item 
If $k=0$, we have $\cX_t^{\pm (b;p|0;h)}= \cX_t^{\pm (b-p)}$. 
\\
If $k=1$, we have $\cX_t^{\pm (b;p|1;h)} = \cX_t^{\pm ((1);h)} \cX_t^{\pm (b-p-1)}$. 

\item 
If $p=b$, we have 
$\cX_t^{\pm (b;b | k;h)} = \begin{cases} 1 & \text{ if } k=0, \\ 0 & \text{ if } k \not=0. \end{cases}$

\item 
If $b,p >0$, we have 
$\cX_t^{\pm (b;p|k;h)}=\cX_t^{\pm (b-1; p-1|k;h)}$. 

\item 
If $b,k >0$, we have 
\begin{align*}
\cX_t^{\pm (b;p|k;h)} = \frac{1}{k} \sum_{z=1}^k z \cX_t^{\pm ((z);h)} \cX_t^{\pm (b-1;p|k-z ; h)}.
\end{align*}

\item 
If $b>0$, we have 
\begin{align*}
(b-p+k) \cX_t^{\pm (b;p|k;h)} = \cX_t^{\pm} \cX_t^{\pm (b-1 ; p | k ;h)} 
	+ \sum_{z=1}^k (z+1) \cX_t^{\pm ((z);h)} \cX_t^{\pm (b-1 ; p | k-z;h)}. 
\end{align*}
\end{enumerate}
\end{lem}

\begin{proof} 
(\roi), (\roii), (\roiii) and (\roiv) are clear from definitions. 

We prove (\rov). 
Note that $\sum_{z \geq 1} z m_z (\la) =k$ for a partition $\la$ of $k$. 
Then, by the definition \eqref{Def cXbpkh}, 
we have 
\begin{align*}
\cX_t^{\pm (b;p|k;h)} 
= \sum_{\la \vdash k} \cX_t^{\pm (\la ; h)} \cX_t^{\pm (b-p-\ell(\la))} 
= \frac{1}{k} \sum_{\la \vdash k} \big( \sum_{z \geq 1} z m_z (\la) \big) \cX_t^{\pm (\la;h)} \cX_t^{\pm (b-p- \ell(\la))}.  
\end{align*}
On the other hand, by the definition \eqref{Def cXlah}, 
we have 
\begin{align*}
\cX_t^{\pm (\la ; h)} 
= \prod_{ j \geq 1} \frac{(\cX_t^{\pm ((j) ; h)})^{m_j (\la)}}{m_j(\la) !} 
= \frac{1}{m_z(\la)} \cX_t^{\pm ((z);h)} \frac{(\cX_t^{\pm ((z);h)})^{m_z(\la) -1}}{(m_z (\la) -1)!} 
 \prod_{j \geq 1\atop  j \not=z} \frac{(\cX_t^{\pm ((j) ; h)})^{m_j (\la)}}{m_j(\la) !}
\end{align*}
for each $z$ such that $m_z(\la) \not=0$. 
Thus, we have 
\begin{align*}
\cX_t^{\pm (b;p|k;h)} 
&= \frac{1}{k} \sum_{\la \vdash k} \sum_{z \geq 1 \atop m_z (\la) \not=0}   
	z \cX_t^{\pm ((z);h)} \frac{(\cX_t^{\pm ((z);h)})^{m_z(\la) -1}}{(m_z (\la) -1)!} 
		 \prod_{j \geq 1\atop  j \not=z} \frac{(\cX_t^{\pm ((j) ; h)})^{m_j (\la)}}{m_j(\la) !}
	\cX_t^{\pm (b-p- \ell(\la))}
\\
&= \frac{1}{k} \sum_{z =1}^k z \cX_t^{\pm ((z);h)} \sum_{\la \vdash k \atop m_z (\la) \not=0} 
	  \frac{(\cX_t^{\pm ((z);h)})^{m_z(\la) -1}}{(m_z (\la) -1)!} 
		 \prod_{j \geq 1\atop  j \not=z} \frac{(\cX_t^{\pm ((j) ; h)})^{m_j (\la)}}{m_j(\la) !}
	\cX_t^{\pm (b-p- \ell(\la))}
\\
&= \frac{1}{k} \sum_{z =1}^k z \cX_t^{\pm ((z);h)}  
	\sum_{\mu \vdash k -z} \prod_{j \geq 1} \frac{(\cX_t^{\pm ((j);h)})^{m_j (\mu)}}{m_j(\mu) !} 
	\cX_t^{\pm (b-p- (\ell(\mu) +1))}
\\
&= \frac{1}{k} \sum_{z =1}^k z \cX_t^{\pm ((z);h)}  
	\sum_{\mu \vdash k -z} \cX_t^{\pm (\mu;h)} \cX_t^{ \pm ( (b -1) -p - \ell(\mu))} 
\\
&= \frac{1}{k} \sum_{z =1}^k z \cX_t^{\pm ((z);h)} \cX_t^{\pm (b-1 ; p | k-z ;h)}. 
\end{align*}

We prove (\rovi). 
By the definition \eqref{Def cXbpkh}, 
we have 
\begin{align*}
&(b-p+k) \cX_t^{\pm (b;p|k;h)} 
\\
&= k \cX_t^{\pm (b;p|k;h)}  
	+  \sum_{\la \vdash k} \ell (\la) \cX_t^{\pm (\la;h)} \cX_t^{\pm (b-p- \ell(\la))} 
	+  \sum_{\la \vdash k}(b-p- \ell (\la)) \cX_t^{\pm (\la;h)} \cX_t^{\pm (b-p- \ell(\la))}. 
\end{align*}
Note that $\ell(\la) = \sum_{z \geq 1} m_z (\la)$, 
$(b-p-\ell(\la)) \cX_t^{\pm (b-p- \ell(\la))} = \cX_t^{\pm} \cX_t^{\pm (b-p- \ell(\la)-1)}$ 
and the defining relation (L4), 
we have 
\begin{align*}
&(b-p+k) \cX_t^{\pm (b;p|k;h)} 
\\
&= k \cX_t^{\pm (b;p|k;h)}  
	+  \sum_{\la \vdash k} \big( \sum_{z \geq 1} m_z (\la) \big) \cX_t^{\pm (\la;h)} \cX_t^{\pm (b-p- \ell(\la))} 
	+  \cX_t^{\pm} \sum_{\la \vdash k} \cX_t^{\pm (\la;h)} \cX_t^{\pm (b -1 -p- \ell(\la))}. 
\end{align*}
In a similar argument as in the proof of (\rov), we have 
\begin{align*}
&(b-p+k) \cX_t^{\pm (b;p|k;h)} 
\\
&= k \cX_t^{\pm (b;p|k;h)}  
	+ \sum_{z=1}^k \cX_t^{\pm ((z);h)} \cX_t^{\pm (b-1 ; p | k-z;h)} 
	+  \cX_t^{\pm} \sum_{\la \vdash k} \cX_t^{\pm (\la;h)} \cX_t^{ \pm ((b -1) -p- \ell(\la))}. 
\end{align*} 
Then, by (\rov) and the definition \eqref{Def cXbpkh}, we have 
\begin{align*}
&(b-p+k) \cX_t^{\pm (b;p|k;h)} 
\\
&=\sum_{z=1}^k z \cX_t^{\pm ((z);h)} \cX_t^{\pm (b-1;p|k-z ;h)} 
	+ \sum_{z=1}^k \cX_t^{\pm ((z);h)} \cX_t^{\pm (b-1 ; p | k-z;h)}
	+ \cX_t^{\pm} \cX_t^{\pm (b-1 ; p | k ;h)} 
\\
&= \cX_t^{\pm} \cX_t^{\pm (b-1 ; p | k ;h)} 
	+ \sum_{z=1}^k (z+1) \cX_t^{\pm ((z);h)} \cX_t^{\pm (b-1 ; p | k-z;h)}. 
\end{align*}
\end{proof}


\begin{lem} 
\label{Lemma cXt+ cXs-cpks+t}
For $s,t,c,p,k \in \ZZ_{\geq 0}$, we have 
\begin{align}
\label{cXt+ cXs-cpks+t}
\begin{split} 
&[\cX_t^+, \cX_s^{-(c;p|k;s+t)}] 
\\
&= 
\sum_{z=0}^k \sum_{w=0}^{k-z} \begin{pmatrix} k-z \\ w \end{pmatrix} (-Q)^w  
		\cX_s^{-(c;p+1|z;s+t)} 	\cJ_{(k-z+1)(s+t) +w}^{\lan 1 \ran} 
	-(k+1) \cX_s^{-(c;p+1|k+1;s+t)}. 
\end{split}
\end{align}
\end{lem}

\begin{proof} 
If $c=0$, the equation \eqref{cXt+ cXs-cpks+t} follows from Lemma \ref{Lemma cXtbpkh} (\roi) and (\roiii). 
Then, 
we prove \eqref{cXt+ cXs-cpks+t} 
by the induction on $k$ in the case where $c>0$.  

If $k=0$, 
we see that 
\eqref{cXt+ cXs-cpks+t} 
is just the formula in Lemma \ref{Lemma cXt cXsc} 
by Lemma \ref{Lemma cXtbpkh} (\roii). 

If $k>0$, 
by Lemma \ref{Lemma cXtbpkh} (\rov) and the defining relation (L4),  
we have 
\begin{align*}
\cX_t^+ \cX_s^{- (c;p|k;s+t)} 
=\frac{1}{k} \sum_{z=1}^k z  \cX_t^+ \cX_s^{- (c-1;p | k-z ; s+t)} \cX_s^{-((z);s+t)}.  
\end{align*}
Applying the assumption of the induction, 
we have 
\begin{align*}
&\cX_t^+ \cX_s^{- (c;p|k;s+t)} 
\\
&= \frac{1}{k} \sum_{z=1}^k z \Big\{ 
	\cX_s^{-(c-1;p | k-z ; s+t)} \cX_t^+ 
	\\ & \hspace{3em} 
	+ \sum_{y=0}^{k-z} \cX_s^{-(c-1; p +1 | y ; s+t)} 
		\Big( \sum_{w=0}^{k-z-y} \begin{pmatrix} k-z-y \\ w \end{pmatrix} (-Q)^w \cJ_{(k-z-y+1)(s+t) +w}^{\lan 1 \ran} \Big) 
	\\ & \hspace{3em}
	- (k-z +1) \cX_s^{-(c-1 ; p+1 | k-z+1 ; s+t)} 
	\Big\} \cX_s^{-((z);s+t)}. 
\end{align*} 
Applying Lemma \ref{Lemma cXt+ cXs-(p)h} and Lemma \ref{Lemma cJs1 cXtph} (\roii), 
we have 
\begin{align*}
&\cX_t^+ \cX_s^{-(c;p|k;s+t)} 
\\
&= \frac{1}{k} \sum_{z=1}^k z \cX_s^{-(c-1;p|k-z;s+t)} 
	\Big( \cX_s^{-((z);s+t)} \cX_t^+ 
		+ \sum_{w=0}^z \begin{pmatrix} z \\ w \end{pmatrix} (-Q)^w \cJ_{(z+1) (s+t) +w}^{\lan 1 \ran} \Big)
	\\ & \quad 
	+ \frac{1}{k} \sum_{z=1}^k \sum_{y=0}^{k-z} z \cX_s^{-(c-1 ; p+1 | y ; s+t)} 
		\sum_{w=0}^{k-z-y} \begin{pmatrix} k-z-y \\ w \end{pmatrix} (-Q)^w 
		\\ & \hspace{2em} \times 
		\Big( \cX_s^{-((z) ; s+t)} \cJ_{(k-z-y+1)(s+t) +w}^{\lan 1 \ran} 
			- 2 \cX_{s + (k-z-y)(s+t) +w }^{-( (z+1) ; s+t)} \Big) 
	\\ & \quad 
	- \frac{1}{k} \sum_{z=1}^k z (k-z+1) \cX_s^{-(c-1 ; p+1 | k - z +1 ; s+t)} \cX_s^{-((z) ; s+t)}. 	
\end{align*}
Put 
\begin{align*}
&(*1) = \frac{1}{k} \sum_{z=1}^k z \cX_s^{-(c-1;p|k-z;s+t)} \cX_s^{-((z);s+t)} \cX_t^+, 
\\
&(*2) = \frac{1}{k} \sum_{z=1}^k z \cX_s^{-(c-1 ; p | k-z ; s+t)} 
	\sum_{w=0}^z \begin{pmatrix} z \\ w \end{pmatrix} (-Q)^w \cJ_{(z+1) (s+t) +w}^{\lan 1 \ran} 
\\
&(*3) = \frac{1}{k} \sum_{z=1}^k \sum_{y=0}^{k-z} z \cX_s^{-(c-1 ; p+1 | y ; s+t)} 
	\sum_{w=0}^{k-z-y} \begin{pmatrix} k-z-y \\ w \end{pmatrix} (-Q)^w 
	\cX_s^{-((z) ; s+t)} \cJ_{(k-z-y+1)(s+t) +w}^{\lan 1 \ran} 
\\
&(*4) = \frac{1}{k} \sum_{z=1}^k \sum_{y=0}^{k-z} z \cX_s^{-(c-1 ; p+1 | y ; s+t)} 
	\sum_{w=0}^{k-z-y} \begin{pmatrix} k-z-y \\ w \end{pmatrix} (-Q)^w 
	\cX_{s + (k-z-y)(s+t) +w}^{-((z+1) ; s+t)}, 
\\
&(*5) = \frac{1}{k} \sum_{z=1}^k z (k-z+1) \cX_s^{-(c-1 ; p+1 | k - z +1 ; s+t)} \cX_s^{-((z) ; s+t)}, 
\end{align*}
then we have 
\begin{align}
\label{cXt+ cXs-cpks+t *}
\cX_t^+ \cX_s^{-(c;p | k ; s+t)} 
= (*1) + (*2) + (*3) - 2 (*4) - (*5). 
\end{align}

By Lemma \ref{Lemma cXtbpkh} (\rov) together with (L4), 
we have 
\begin{align} 
\label{*1} 
(*1) =  \cX_s^{-(c ; p | k ; s+t)} \cX_t^+. 
\end{align}

Put $z'=k-z$ in $(*2)$ and apply Lemma \ref{Lemma cXtbpkh} (\roiv),  
we have 
\begin{align}
\label{*2}
(*2) = \frac{1}{k} \sum_{z'=0}^{k-1} (k-z') \cX_s^{-(c ; p +1| z' ; s+t)} 
	\sum_{w=0}^{k-z'} \begin{pmatrix} k-z' \\ w \end{pmatrix} (-Q)^w \cJ_{(k-z'+1)(s+t) +w}^{\lan 1 \ran}. 
\end{align}

Put $h= z+y$ in $(*3)$ , we have 
\begin{align*}
(*3) = \frac{1}{k} \sum_{h=1}^k \sum_{z=1}^h z \cX_s^{-(c-1 ; p+1 | h-z ; s+t)} \cX_s^{-((z);s+t)} 
	\sum_{w=0}^{k- h} \begin{pmatrix} k- h \\ w \end{pmatrix} (-Q)^w \cJ_{ (k - h +1)(s+t) +w}^{\lan 1 \ran}. 
\end{align*} 
Applying Lemma \ref{Lemma cXtbpkh} (\rov) together with (L4), 
we have 
\begin{align}
\label{*3} 
(*3) = \frac{1}{k} \sum_{h=1}^k h \cX_s^{-(c ; p+1 | h ; s+t)} 
	\sum_{w=0}^{k-h} \begin{pmatrix} k- h \\ w \end{pmatrix} (-Q)^w \cJ_{ (k - h +1)(s+t) +w}^{\lan 1 \ran}. 
\end{align}  

By \eqref{*2} and \eqref{*3}, we have 
\begin{align}
\label{*2 + *3} 
(*2) + (*3) 
= \sum_{z=0}^k \sum_{w=0}^{k-z} \begin{pmatrix} k-z \\ w \end{pmatrix} (-Q)^w 
	\cX_s^{-(c ; p+1 | z ; s+t)}  \cJ_{ (k-z+1)(s+t) +w}^{\lan 1 \ran}. 
\end{align}

We also have 
\begin{align*}
(*4) 
=\frac{1}{k} \sum_{y=0}^{k-1} \sum_{z=1}^{k-y} z \cX_s^{-(c-1 ; p+1 | y ; s+t)} 
	\sum_{w=0}^{k-z-y} \begin{pmatrix} k-z-y \\ w \end{pmatrix} (-Q)^w \cX_{s+(k-z-y)(s+t) +w}^{-((z+1);s+t)}. 
\end{align*}
Put $h= k-y+1$, we have 
\begin{align*}
(*4) = \frac{1}{k} \sum_{h=2}^{k+1} \sum_{z=1}^{h-1} z \cX_s^{-(c-1 ; p+1 | k-h+1 ; s+t)} 
	\sum_{w=0}^{h-z-1} \begin{pmatrix} h-z-1 \\ w \end{pmatrix} (-Q)^w \cX_{s + (h-z-1)(s+t)+w}^{-((z+1); s+t)}. 
\end{align*}
Put 
\begin{align*}
(\sharp) = \sum_{w=0}^{h-z-1} \begin{pmatrix} h-z-1 \\ w \end{pmatrix} (-Q)^w \cX_{s + (h-z-1)(s+t)+w}^{-((z+1); s+t)}. 
\end{align*}
By \eqref{Def cX ph},  
we have 
\begin{align*}
(\sharp) 
=\sum_{w=0}^{h-z-1} \begin{pmatrix} h-z-1 \\ w \end{pmatrix} (-Q)^w 
	\sum_{y=0}^{z+1} \begin{pmatrix} z+1 \\ y \end{pmatrix} (-Q)^y \cX_{s+h (s+t) +w +y}^-. 
\end{align*}
Put $y'= w+y$, 
we have 
\begin{align*}
(\sharp) = \sum_{y'=0}^h 
	\Big( \sum_{ w = \max \{ 0, y' - (z+1)\}}^{\min \{ h - (z+1), y' \}} 
		\begin{pmatrix} h- (z +1) \\ w \end{pmatrix} \begin{pmatrix} z+1 \\ y' -w \end{pmatrix} \Big) 
	 (-Q)^{y'} \cX_{s + h(s+t) + y'}^-. 
\end{align*}
Note that 
$ \sum_{ w = \max \{ 0, y' - (z+1)\}}^{\min \{ h - (z+1), y' \}}
	\begin{pmatrix} h- (z +1) \\ w \end{pmatrix} \begin{pmatrix} z+1 \\ y' -w \end{pmatrix} 
	= \begin{pmatrix} h \\ y' \end{pmatrix} $ 
by \eqref{sum max min}, 
we have 
\begin{align*}
(\sharp) 
= \sum_{y'=0}^h \begin{pmatrix} h \\ y ' \end{pmatrix} (-Q)^{y'} \cX_{s+ h (s+t) +y'}^- 
= \cX_s^{-((h); s+t)}. 
\end{align*}  
(Use \eqref{Def cX ph} again.) 
Then, we have 
\begin{align*}
(*4) 
&= \frac{1}{k} \sum_{h=2}^{k+1} \Big( \sum_{z=1}^{h-1} z \Big) \cX_s^{-(c-1 ; p+1 | k-h+1 ; s+t)} \cX_s^{-((h); s+t)} 
\\
&=  \frac{1}{k} \sum_{h=2}^{k+1} \frac{h(h-1)}{2} \cX_s^{-(c-1 ; p+1 | k-h+1 ; s+t)} \cX_s^{-((h); s+t)}. 
\end{align*}
Then we have 
\begin{align}
\label{2 (*4) + (*5)}
\begin{split} 
2 (*4) + (*5) 
&=  \sum_{z=1}^{k+1} z  \cX_s^{-(c-1 ; p+1 | k -z +1 ; s+t)} \cX_s^{-((z); s+t)} 
\\
&= (k+1) \cX_s^{- (c ; p+1 | k+1  ; s+t)}, 
\end{split}
\end{align}
where the last equation follows from Lemma \ref{Lemma cXtbpkh} (\rov).

By \eqref{cXt+ cXs-cpks+t *}, \eqref{*1}, \eqref{*2 + *3} and \eqref{2 (*4) + (*5)}, 
we have 
\begin{align*}
&\cX_t^+ \cX_s^{-(c; p | k ; s+t)} 
\\
&= \cX_s^{-(c;p | k; s+t)} \cX_t^+  
	+ \sum_{z=0}^k \sum_{w=0}^{k-z} \begin{pmatrix} k - z \\ w \end{pmatrix} (-Q)^w 
		\cX_s^{-(c ; p+1 | z ; s+t)}  \cJ_{(k-z+1)(s+t) +w}^{\lan 1 \ran} 
	\\ & \quad 
	-  (k+1) \cX_s^{- (c ; p+1 | k+1;s+t)}. 
\end{align*}
\end{proof}
　
\begin{prop}
\label{Prop comm rel cXt+b cXs-c}
For $s,t,b,c \in \ZZ_{ \geq 0}$, we have 
\begin{align}
\label{Prop cXtb cXsc}
[\cX_t^{+(b)}, \cX_s^{-(c)}] 
= \sum_{p=1}^{\min \{b,c\}} \sum_{k=0}^p \sum_{l=0}^{p-k} 
	(-1)^{k+l} \cX_s^{-(c;p|k;s+t)} \cJ_{s+t}^{\lan p- (k+l) \ran} \cX_t^{+(b;p|l;s+t)}. 
\end{align}
\end{prop}

\begin{proof}
We prove \eqref{Prop cXtb cXsc} by the induction on $b$. 
If $b=1$, 
\eqref{Prop cXtb cXsc} follows from Lemma \ref{Lemma cXt cXsc} together with Lemma \ref{Lemma cXtbpkh}. 

If $b >1$, we have 
\begin{align*}
&\cX_t^{+(b)} \cX_s^{-(c)} 
\\
&= \frac{1}{b} \cX_t^+ \cX_t^{(b-1)} \cX_s^{-(c)} 
\\
&= \frac{1}{b} \cX_t^+ 
	\Big( \cX_s^{-(c)} \cX_t^{+(b-1)} 
		+ \sum_{p=1}^{\min \{ b-1, c \}} \sum_{k=0}^p \sum_{l=0}^{p-k} 
		(-1)^{k+l} \cX_s^{-(c;p | k ; s+t)} \cJ_{s+t}^{\lan p - (k+l)\ran} \cX_t^{+(b -1 ; p | l ; s+t)} 
	\Big) 
\end{align*}
by the assumption of the induction. 
Applying Lemma \ref{Lemma cXt cXsc} and Lemma \ref{Lemma cXt+ cXs-cpks+t}, 
we have 
\begin{align*}
&\cX_t^{+(b)} \cX_s^{-(c)} 
\\
&= \frac{1}{b} \Big\{ 
	\Big( \cX_s^{-(c)} \cX_t^+ + \cX_s^{-(c-1)} \cJ_{s+t}^{\lan 1 \ran} - \cX_s^{-(c-2)} \cX_s^{-((1) ; s+t)} 
	\Big) \cX_t^{+(b-1)} 
	\\ & \hspace{3em} 
	+ \sum_{p=1}^{ \min \{ b-1, c\}} \sum_{k=0}^p \sum_{l=0}^{p-k} (-1)^{k+l} 
		\\ & \hspace{4em} \times 
		\Big( \cX_s^{-(c ; p | k ; s+t)} \cX_t^+ 
			+ \sum_{z=0}^k \sum_{w=0}^{k-z} \begin{pmatrix} k-z \\ w \end{pmatrix} (-Q)^w 
				\cX_s^{-(c; p+1 | z ; s+t)} \cJ_{(k-z+1)(s+t) +w}^{\lan 1 \ran} 
			\\ & \hspace{6em} 
			- (k+1) \cX_s^{-(c;p+1 | k+1 ; s+t)} 
		\Big) \cJ_{s+t}^{\lan p - (k+l) \ran} \cX_{t}^{+(b-1 ; p | l ; s+t)} 
	\Big\}. 
\end{align*}
On the other hand, 
by Lemma \ref{Lemma cJsp cXt}, we have 
\begin{align*}
&\cX_s^{-(c ; p | k ; s+t)} \cX_t^+ \cJ_{s+t}^{\lan p - (k+l) \ran} \cX_t^{+(b-1 ; p | l ; s+t)} 
\\
&= \cX_s^{-(c ; p | k ; s+t)} 
	\Big( \cJ_{s+t}^{\lan p - (k+l) \ran} \cX_t^+ 
		- \sum_{z=1}^{p - (k+l)} (-1)^{z+1} (z+1) \cJ_{s+t}^{\lan p -(k+l) - z \ran} \cX_t^{+((z); s+t)} 
	\Big) \cX_t^{+(b-1 ; p | l ; s+t)}. 
\end{align*}
Put 
\begin{align*}
& (*1) = b \cX_s^{-(c)} \cX_t^{+(b)} + \cX_s^{-(c-1)} \cJ_{s+t}^{\lan 1 \ran} \cX_t^{+(b-1)} 
	- \cX_s^{-(c-2)} \cX_s^{-((1); s+t)} \cX_t^{+(b-1)}, 
\\
& (*2) = \sum_{p=1}^{\min \{ b,c\}} \sum_{k=0}^p \sum_{l=0}^{p-k} (-1)^{k+l} 
	\cX_s^{-(c ; p | k ; s+t)} \cJ_{s+t}^{\lan p -(k+l) \ran} \cX_t^+ \cX_t^{+(b-1 ; p | l ; s+t)}, 
\\
&(*3) = \sum_{p=1}^{\min \{ b,c\}} \sum_{k=0}^p \sum_{l=0}^{p-k} \sum_{z=1}^{p- (k+l)} (-1)^{k+l +z} (z+1)
	\\ & \hspace{5em} \times 
	\cX_s^{-(c ; p | k ; s+t)} \cJ_{s+t}^{\lan p - (k+l) -z \ran} \cX_t^{+((z) ; s+t)} \cX_t^{+(b-1 ; p | l ; s+t)}, 
\\
&(*4) = \sum_{p=1}^{\min \{ b,c\}} \sum_{k=0}^p \sum_{l=0}^{p-k} \sum_{z=0}^k \sum_{w=0}^{k-z} (-1)^{k+l} 
	 \begin{pmatrix} k-z \\ w \end{pmatrix} (-Q)^w 
	 \\ & \hspace{5em} \times 
	\cX_s^{-(c; p+1 | z ; s+t)} \cJ_{(k-z+1)(s+t) +w}^{\lan 1 \ran} \cJ_{s+t}^{\lan p - (k+l) \ran} \cX_{t}^{+(b-1 ; p | l ; s+t)}, 
\\
&(*5)= \sum_{p=1}^{\min \{ b,c\}} \sum_{k=0}^p \sum_{l=0}^{p-k} (-1)^{k+l+1} (k+1) 
	\cX_s^{-(c;p+1 | k+1 ; s+t)} \cJ_{s+t}^{\lan p - (k+l) \ran} \cX_{t}^{+(b-1 ; p | l ; s+t)}, 
\end{align*}
then we have 
\begin{align}
\label{cXtn cXsc}
\cX_t^{+(b)} \cX_s^{-(c)} 
= \frac{1}{b} \big\{ (*1) + (*2) + (*3) + (*4) + (*5) \big\}, 
\end{align}
where we note that 
$\cX_t^{+(b-1 ; p | l ; s+t)} =0$ if $p=b$ by Lemma \ref{Lemma cXtbpkh} (\roi). 

By Lemma \ref{Lemma cXtbpkh} (\roii) together with (L4), 
we have 
\begin{align}
\label{**1}
(*1) = b \cX_s^{-(c)} \cX_t^{+(b)} 
	+ \cX_s^{-(c  ; 1 | 0 ; s+t)} \cJ_{s+t}^{\lan 1 \ran} \cX_t^{+(b;1 |0 ;s+t)} 
	- \cX_s^{-(c;1 | 1 ; s+t)} \cJ_{s+t}^{\lan 0 \ran} \cX_t^{+(b ; 1 | 0 ;s+t)}. 
\end{align}

Put $l'=l+z$ in $(*3)$,  
we have 
\begin{align*}
&(*3) 
\\
&= \sum_{p=1}^{\min \{ b,c \}} \sum_{k=0}^p \sum_{l'=1}^{p-k} \sum_{z=1}^{l'} (-1)^{k + l'} (z+1) 
	\cX_s^{-(c ; p | k ; s+t)} \cJ_{s+t}^{\lan p - k - l' \ran} \cX_t^{+((z); s+t)} \cX_t^{+(b-1 ; p | l'-z ; s+t)}. 
\end{align*}
Then we have 
\begin{align*}
(*2) + (*3) 
&=\sum_{p=1}^{\min\{ b, c\}} \sum_{k=0}^p \sum_{l=0}^{p-k} 
	(-1)^{k+l} \cX_s^{-(c;p |k ; s+t)} \cJ_{s+t}^{\lan p - (k+l) \ran} 
	\\ & \hspace{2em} \times 
	\Big( \cX_t^+ \cX_t^{+(b-1 ; p | l ; s+t)} 
		+ \sum_{z=1}^l (z+1) \cX_t^{+((z); s+t)} \cX_t^{+(b-1 ; p | l-z ; s+t)} \Big), 
\end{align*}
where we note that $\sum_{z=1}^0 \cX_t^{+((z); s+t)} \cX_t^{+(b-1 ; p | l-z ; s+t)} =0$. 
Applying Lemma \ref{Lemma cXtbpkh} (\rovi), 
we have  
\begin{align}
\label{**2 + **3}
\begin{split} 
(*2) + (*3) 
&= \sum_{p=1}^{\min \{b,c \}} \sum_{k=0}^p \sum_{l=0}^{p-k} (-1)^{k+l} (b-p+l) 
	\cX_s^{-(c;p |k ; s+t)} \cJ_{s+t}^{\lan p - (k+l) \ran} \cX_t^{+(b;p | l ; s+t)} 
\\
&= (b-1) \cX_s^{-(c ; 1 | 0 ; s+t)} \cJ_{s+t}^{\lan 1 \ran} \cX_t^{+(b ; 1 | 0 ; s+t)} 
	- b \cX_s^{-(c ; 1 | 0 ; s+t)} \cJ_{s+t}^{\lan 0 \ran} \cX_t^{+(b ; 1 | 1 ; s+t)}
	\\ & \hspace{1em} 
	- (b-1) \cX_s^{-(c ; 1 | 1 ; s+t)} \cJ_{s+t}^{\lan 0 \ran} \cX_t^{+(b ; 1 | 0; s+t)} 
	\\ & \hspace{1em} 
	+ \sum_{p=2}^{\min \{b,c \}} \sum_{k=0}^p \sum_{l=0}^{p-k} (-1)^{k+l} (b-p+l) 
	\cX_s^{-(c;p |k ; s+t)} \cJ_{s+t}^{\lan p - (k+l) \ran} \cX_t^{+(b;p | l ; s+t)}. 
\end{split} 
\end{align}

Put $p'= p+1$ in $(*4)$, 
we have 
\begin{align*}
(*4) 
&= \sum_{p'=2}^{\min\{ b, c \}} \sum_{k=0}^{p'-1} \sum_{l=0}^{p'-k-1} \sum_{z=0}^k \sum_{w=0}^{k-z} 
	(-1)^{k+l} \begin{pmatrix} k - z \\ w \end{pmatrix} (-Q)^w 
	\\ & \hspace{3em} \times 
	\cX_s^{-(c ; p' | z ; s+t)} \cJ_{(k-z+1)(s+t) +w}^{\lan 1 \ran} \cJ_{s+t}^{\lan p' - (k+l) -1 \ran} \cX_t^{+(b-1; p'-1 | l ; s+t)}, 
\end{align*} 
where we note that 
$\cX_t^{+(b-1 ; p'-1 | l ; s+t)} =0$ if $p'=b+1$, 
and 
$\cX_s^{-(c;p' | z; s+t)} =0 $ if $p'=c+1$ 
by Lemma \ref{Lemma cXtbpkh} (\roi). 
Note that 
\begin{align*} 
\sum_{k=0}^{p-1} \sum_{l=0}^{p-k-1} \sum_{z=0}^k 
= \sum_{z=0}^{p-1} \sum_{k=z}^{p-1} \sum_{l=0}^{p-k-1} 
=  \sum_{z=0}^{p-1} \sum_{l=0}^{p-z-1} \sum_{k=z}^{p-l-1}, 
\end{align*}
we have 
\begin{align*}
(*4) 
&= \sum_{p=2}^{\min\{ b,c\}} \sum_{z=0}^{p-1} \sum_{l=0}^{p-z-1} \cX_s^{-(c ; p | z ; s+t)} 
	\\ & \hspace{2em} \times 
	\Big( \sum_{k=z}^{p-l-1} \sum_{w=0}^{k-z}(-1)^{k+l} \begin{pmatrix} k - z \\ w \end{pmatrix} (-Q)^w 
	 \cJ_{(k-z+1)(s+t) +w}^{\lan 1 \ran} \cJ_{s+t}^{\lan p - (k+l) -1 \ran} 
	 \Big)  \cX_t^{+(b-1; p -1 | l ; s+t)}. 
\end{align*} 
Put $k'=k-z+1$, 
we have 
\begin{align*}
&\sum_{k=z}^{p-l-1} \sum_{w=0}^{k-z}(-1)^{k+l} \begin{pmatrix} k - z \\ w \end{pmatrix} (-Q)^w 
	 \cJ_{(k-z+1)(s+t) +w}^{\lan 1 \ran} \cJ_{s+t}^{\lan p - (k+l) -1 \ran} 
\\
&= \sum_{k'=1}^{p-l-z} \sum_{w=0}^{k'-1} (-1)^{k'+z +l-1} \begin{pmatrix} k'-1 \\ w \end{pmatrix} (-Q)^w 
	\cJ_{ k' (s+t) +w}^{\lan 1 \ran} \cJ_{s+t}^{\lan p - k' -z -l \ran}.  
\end{align*}
Since $\cJ_{k'(s+t) +w}^{\lan 1 \ran} = \cJ_{k'(s+t) +w} + (-Q) \cJ_{k'(s+t)+w+1}$, 
we see that 
\begin{align*}
\sum_{w=0}^{k'-1} \begin{pmatrix} k'-1 \\ w \end{pmatrix} (-Q)^w \cJ_{k'(s+t) +w}^{\lan 1 \ran} 
= \sum_{w=0}^{k'} \begin{pmatrix} k' \\ w \end{pmatrix} (-Q)^w \cJ_{k'(s+t)+w}. 
\end{align*}
Thus we have 
\begin{align*}
&\sum_{k=z}^{p-l-1} \sum_{w=0}^{k-z}(-1)^{k+l} \begin{pmatrix} k - z \\ w \end{pmatrix} (-Q)^w 
	 \cJ_{(k-z+1)(s+t) +w}^{\lan 1 \ran} \cJ_{s+t}^{\lan p - (k+l) -1 \ran} 
\\
&= (-1)^{z+l} \sum_{k'=1}^{p-l-z} (-1)^{k' -1} \sum_{w=0}^{k'} \begin{pmatrix} k' \\ w \end{pmatrix} (-Q)^w 
	\cJ_{k' (s+t) +w}  \cJ_{s+t}^{\lan p - k' - z- l \ran} 
\\
&= (-1)^{z+l} (p-l-z) \cJ_{s+t}^{\lan p-l-z \ran}, 
\end{align*}
where the last equation follows from \eqref{Def cJ p}. 
Then we have 
\begin{align*} 
(*4) 
&= \sum_{p=2}^{\min \{ b , c \}} \sum_{z=0}^{p-1} \sum_{l=0}^{p-z-1} 
	(p-l-z) (-1)^{z+l}  \cX_s^{-(c ; p | z  ; s+t)} \cJ_{s+t}^{\lan p - l - z \ran} \cX_t^{+(b-1 ; p-1 | l ; s+t)} 
\\
&= \sum_{p=2}^{\min \{ b , c \}} \sum_{z=0}^{p} \sum_{l=0}^{p-z} 
	(p-l-z) (-1)^{z+l}  \cX_s^{-(c ; p | z  ; s+t)} \cJ_{s+t}^{\lan p - l - z \ran} \cX_t^{+(b-1 ; p-1 | l ; s+t)}. 
\end{align*}
Applying Lemma \ref{Lemma cXtbpkh} (\roiv), 
we have 
\begin{align}
\label{**4}
(*4) = \sum_{p=2}^{\min \{ b , c \}} \sum_{z=0}^{p} \sum_{l=0}^{p-z} 
	(p-l-z) (-1)^{z+l}  \cX_s^{-(c ; p | z  ; s+t)} \cJ_{s+t}^{\lan p - l - z \ran} \cX_t^{+(b ; p | l ; s+t)}. 
\end{align}

Put $p'=p+1$ in $(*5)$, 
we have 
\begin{align*} 
(*5)= 
\sum_{p'=2}^{\min \{b , c \}} \sum_{k = 0}^{p'-1} \sum_{l=0}^{p'-k-1} (-1)^{k+l+1} (k+1) 
	\cX_s^{-(c ; p' | k+1 ; s+t)} \cJ_{s+t}^{\lan p' - k - l -1 \ran} \cX_t^{+(b-1 ; p'-1 | l ; s+t)}, 
\end{align*}
where we note that 
$\cX_t^{+(b-1 ; p'-1 | l ; s+t)} =0$ if $p'=b+1$, 
and 
$\cX_s^{-(c;p' | k+1; s+t)} =0 $ if $p'=c+1$ 
by Lemma \ref{Lemma cXtbpkh} (\roi). 
Put $k'=k+1$, 
we have 
\begin{align*}
(*5) 
&= \sum_{p'=2}^{\min \{ b,c \}} \sum_{k'=1}^{p'} \sum_{l=0}^{p' - k'} (-1)^{k' +l} k' 
	\cX_s^{-(c ; p' | k' ; s+t)} \cJ_{s+t}^{\lan p'-k' -l \ran} \cX_t^{+(b-1 ; p'-1 | l ; s+t)} 
\\
&= \sum_{p=2}^{\min \{ b , c \}} \sum_{k=0}^p \sum_{l=0}^{p-k} k (-1)^{k+l} 
	\cX_s^{-(c ; p | k ; s+t)} \cJ_{s+t}^{\lan p-k -l \ran} \cX_t^{+(b-1 ; p-1 | l ; s+t)}. 
\end{align*}
Applying Lemma \ref{Lemma cXtbpkh} (\roiv), 
we have 
\begin{align}
\label{**5}
(*5) 
= \sum_{p=2}^{\min \{ b , c \}} \sum_{k=0}^p \sum_{l=0}^{p-k} k (-1)^{k+l} 
	\cX_s^{-(c ; p | k ; s+t)} \cJ_{s+t}^{\lan p-k -l \ran} \cX_t^{+(b ; p | l ; s+t)}.
\end{align}

By \eqref{cXtn cXsc}, \eqref{**1}, \eqref{**2 + **3}, \eqref{**4} and \eqref{**5}, 
we have 
\begin{align*}
\cX_t^{+(b)} \cX_s^{-(c)} 
= \cX_s^{-(c)} \cX_t^{+(b)} 
	+ \sum_{p=1}^{\min \{b , c \}} \sum_{k=0}^p  \sum_{l=0}^{p-k} (-1)^{k+l} 
		\cX_s^{-(c ; p | k ; s+t)} \cJ_{s+t}^{\lan p - (k+l) \ran} \cX_t^{+(b ; p | l ; s+t)}. 
\end{align*}
\end{proof}



\section{Rank $1$ case ; finite dimensional simple modules of  $U (\Fsl_2^{\lan Q \ran} [x])$} 

In this section, we classify the finite dimensional simple $U(\Fsl_2^{\lan Q \ran}[x])$-modules. 

\para {\bf $1$-dimensional representations.}
\label{1-dim rep sl2}
First, we consider  $1$-dimensional representations of $\Fsl_2^{\lan Q \ran}[x]$.  
Let $L= \CC v$ be a $1$-dimensional $U(\Fsl_2^{\lan Q \ran}[x])$-module with a basis $\{v\}$, 
then $\cJ_t$ ($t \geq 0$) acts on $v$ as a scalar multiplication. 
If $\cX_t^+ \cdot v \not=0$ (resp. $\cX_t^- \cdot v \not=0$), 
then $\cX_t^+ \cdot v$ (resp. $\cX_t^- \cdot v$) is an eigenvector for the action of $\cJ_0$ 
whose eigenvalue is different from one of $v$ by the defining relation (L2). 
This is a contradiction since $L$ is $1$-dimensional. 
Thus, we have $\cX_t^{\pm} \cdot v=0$  for $t \geq 0$. 
Moreover, by the defining relation (L3), we have 
$(\cJ_t - Q \cJ_{t+1}) \cdot v = (\cX_t^+ \cX_0^- - \cX_0^- \cX_t^+) \cdot v =0$. 
This implies that $\cJ_t \cdot v=0$ for $t \geq 0$ if $Q=0$, 
and that $\cJ_t \cdot v = Q^{-t} \cJ_{0} \cdot v $ for $t >0$ if $Q \not=0$. 

We define the set $\BB^{\lan Q \ran}$ by 
\begin{align*}
\BB^{\lan Q \ran} = \begin{cases}
		\{0 \} & \text{ if } Q=0, 
		\\
		\CC & \text{ if } Q\not=0. 
		\end{cases} 
\end{align*}
For each $\b \in \BB^{\lan Q \ran}$, 
we can define the $1$-dimensional $U(\Fsl_2^{\lan Q \ran}[x])$-module $\cL^\b = \CC v_0$ such that 
\begin{align*}
\cX_t^{\pm} \cdot v_0 =0, \quad 
\cJ_t \cdot v_0 = 
	\begin{cases} 
		0 & \text{ if } Q=0, 
		\\
		Q^{-t} \b v_0 & \text{ if } Q \not=0
	\end{cases} 
\quad (t \in \ZZ_{\geq 0}) 
\end{align*}
by checking the defining relations of $\Fsl_2^{\lan Q \ran}[x]$. 
Note that $\cL^0$ is the trivial representation. 
Now we obtain the following lemma. 


\begin{lem} 
\label{Lemma 1-dim rep rank1}
Any $1$-dimensional $U(\Fsl_2^{\lan Q \ran}[x])$-module is isomorphic to $\cL^\b$ for some $\b \in \BB^{\lan Q \ran}$. 
\end{lem}


\para 
Recall from \S \ref{section Rep slmx}, 
a finite dimensional simple $U(\Fsl_2^{\lan Q \ran}[x])$-module is isomorphic to a simple highest weight module 
$\cL(\bu)$ for some highest weight $\bu=(u_t) \in \prod_{t \geq 0} \CC$ (Proposition \ref{Prop simple slmx HW}), 
where we omit the first index for the highest weight. 
Then, in order to classify the finite dimensional simple  $U(\Fsl_2^{\lan Q \ran}[x])$-module, 
it is enough to classify the highest weight $\bu$ such that $\cL(\bu)$ is finite dimensional. 

In order to obtain a necessary condition for $\bu$ such that $\cL(\bu)$ is finite dimensional, 
we prepare the following lemma.  


\begin{lem} 
\label{Lemma sum J v}
Let $M$ be a finite dimensional $U(\Fsl_2^{\lan Q \ran}[x])$-module. 
Take an element $v \in M$ satisfying 
\begin{align*}
\cX_t^+ \cdot v=0, \quad \cJ_t \cdot v = u_t v \quad (t \in \ZZ_{\geq 0}), 
\quad \cX_0^{-(n)} \cdot v \not=0 
\text{ and } 
\cX_0^{-(n+1)} \cdot v=0  
\end{align*}
for some $u_t \in \CC$ ($t \in \ZZ_{\geq 0}$) and $n \in \ZZ_{\geq 0}$.  
(In fact,  a such  element exists by Lemma \ref{Lemma finite dim module}.) 
Then, for $s,t \in \ZZ_{\geq 0}$, we have 
\begin{align*}
\sum_{w=0}^n \begin{pmatrix} n \\ w \end{pmatrix} (-Q)^w \cJ_{t + n s + w}^{\lan 1 \ran} \cdot v 
= \sum_{k=0}^{n-1} (-1)^{n-k+1} 
	\Big( \sum_{w=0}^k \begin{pmatrix} k \\ w \end{pmatrix} (-Q)^w \cJ_{t+ ks +w}^{\lan 1 \ran} \Big) 
	\cJ_s^{\lan n-k \ran} \cdot v. 
\end{align*}
\end{lem}

\begin{proof} 
By the assumption $\cX_0^{-(n+1)} \cdot v=0$ and Proposition \ref{Prop comm rel cXt+b cXs-c}, 
we have 
\begin{align}
\label{0=cX+n cX-n+1 v}
\begin{split} 
0 
&= \cX_s^{+(n)} \cX_0^{-(n+1)} \cdot v 
\\
&= \Big( \cX_0^{-(n+1)} \cX_s^{+(n)} + \sum_{p=1}^n \sum_{k=0}^p \sum_{l=0}^{p-k} 
	(-1)^{k+l} \cX_0^{-(n+1;p|k;s)} \cJ_s^{\lan p - (k+l) \ran} \cX_s^{+(n;p|l;s)} \Big) \cdot v. 
\end{split}
\end{align}
By the definition, we have $\cX_s^{+(n;p|l;s)} = \sum_{\la \vdash l} \cX_s^{+(\la;s)} \cX_s^{+(n-p- \ell(\la))}$. 
Thus, by the definition of $\cX_s^{+(\la;s)}$ and the assumption $\cX_t^+ \cdot v =0$ ($t \geq 0$), 
we have  
\begin{align*}
\cX_s^{+(n;p|l;s)} \cdot v = \begin{cases} v & \text{ if } l=0 \text{ and } p=n, \\ 0 & \text{ otherwise.} \end{cases}
\end{align*}
Then \eqref{0=cX+n cX-n+1 v} implies that 
\begin{align*}
0 = \sum_{k=0}^n (-1)^k \cX_0^{-(n+1;n|k;s)} \cJ_s^{\lan n-k \ran} \cdot v.
\end{align*} 
By the definition, we have 
\begin{align*}
\cX_0^{-(n+1;n|k;s)} 
= \sum_{\la \vdash k} \cX_0^{-(\la;s)} \cX_0^{-(1 - \ell(\la))} 
= \begin{cases} 
		\cX_0^- & \text{ if } k=0, 
		\\
		\cX_0^{-((k);s)} & \text{ if } k \not=0.
	\end{cases}
\end{align*}
Thus, we have 
\begin{align*}
0= \cX_0^- \cJ_s^{\lan n \ran} \cdot v + \sum_{k=1}^n (-1)^k \cX_0^{-((k);s)} \cJ_s^{\lan n-k \ran} \cdot v. 
\end{align*}
By multiplying $\cX_t^+$ from left to this equation, we have 
\begin{align*}
0 
&= \cX_t^+ \cX_0^- \cJ_s^{\lan n \ran} \cdot v + \sum_{k=1}^n (-1)^k \cX_t^+ \cX_0^{-((k);s)} \cJ_s^{\lan n-k \ran} \cdot v 
\\
&= \sum_{k=0}^n (-1)^k 
	\Big( \sum_{w=0}^k \begin{pmatrix} k \\ w \end{pmatrix}(-Q)^w  \cJ_{t+ks+w}^{\lan 1 \ran}  \Big) 
	\cJ_s^{\lan n-k \ran} \cdot v, 
\end{align*}
where we use Lemma \ref{Lemma cXt+ cXs-(p)h} and the fact $\cX_t^+ \cJ_s^{\lan n-k \ran} \cdot v=0$. 
This implies the Lemma.
\end{proof}

This Lemma implies the following proposition 
which gives a necessary condition for $\bu$ such that $\cL(\bu)$ is finite dimensional. 


\begin{prop} 
\label{Prop ne cond rank1}
Let $M$ be a finite dimensional $U(\Fsl_2^{\lan Q \ran}[x])$-module. 
Take an element $v \in M$ satisfying 
\begin{align*}
\cX_t^+ \cdot v=0, \quad \cJ_t \cdot v = u_t v \quad (t \in \ZZ_{\geq 0}), 
\quad \cX_0^{-(n)} \cdot v \not=0 
\text{ and } 
\cX_0^{-(n+1)} \cdot v=0  
\end{align*}
for some $u_t \in \CC$ ($t \in \ZZ_{\geq 0}$) and $n \in \ZZ_{\geq 0}$.  
\begin{enumerate} 
\item 
If $Q=0$, we have $u_0=n$, and there exist $\g_1,\g_2, \dots, \g_n \in \CC$ such that 
\begin{align*}
u_t = p_t (\g_1, \g_2, \dots, \g_n) \quad (t >0),
\end{align*}
where $p_t(\g_1, \dots, \g_n)= \g_1^t + \g_2^t+ \dots + \g_n^t$. 

\item 
If $Q \not=0$, there exist $\b, \g_1,\g_2,\dots, \g_n \in \CC$ such that 
\begin{align*}
u_0 = n + \b \text{ and }  u_t = p_t (\g_1,\g_2,\dots, \g_n) + Q^{-t} \b \quad (t >0), 
\end{align*}
where $p_t(\g_1, \dots, \g_n)= \g_1^t + \g_2^t+ \dots + \g_n^t$. 
\end{enumerate}
\end{prop}

\begin{proof} 
(\roi). 
Assume that $Q=0$. 
Then, $\Fsl_2^{\lan 0 \ran}[x]$ coincides with   the current Lie algebra $\Fsl_2[x]$ of $\Fsl_2$. 
Moreover, the Lie subalgebra of $\Fsl_2 [x]$ generated by $\cX_0^{\pm}$ and $\cJ_0$ 
is isomorphic to $\Fsl_2$. 
Thus, by the representation theory of $\Fsl_2$, 
we have $u_0=n$. 

For $u_1,\dots, u_n$, there exist $\g_1, \dots , \g_n \in \CC$ such that 
\begin{align}
\label{uk 1 to n Q=0}
u_k = p_k (\g_1,\dots, \g_n) 
\text{ for } k=1,\dots, n 
\end{align} 
by Lemma \ref{Lemma solution equations}. 

By the definition, we have 
\begin{align}
\label{cJ1 k} 
\cJ_1^{\lan k \ran} = \frac{1}{k} \sum_{z=1}^k (-1)^{z-1} \cJ_{z} \cJ_1^{\lan k-z \ran}
\end{align} 
since we assume $Q=0$. 
By the induction on $k$ together with \eqref{uk 1 to n Q=0}, \eqref{cJ1 k} and \eqref{e sum e p}, 
we can show that 
\begin{align}
\label{cJi k ek Q=0}
\cJ_1^{\lan k \ran} \cdot v = e_k (\g_1, \dots, \g_n) v 
\text{ for } k=1,\dots, n, 
\end{align} 
where $e_k (\g_1, \dots, \g_n )= \sum_{1 \leq i_1 < i_2 < \dots < i_k \leq n} \g_{i_1} \g_{i_2} \dots \g_{i_k}$. 

By the induction on $t$, we prove that 
\begin{align}
\label{ut = pt Q=0}
u_t = p_t(\g_1,\dots, \g_n) \quad (t>0). 
\end{align}
If $t \leq n$, 
\eqref{ut = pt Q=0} follows from \eqref{uk 1 to n Q=0}. 
If $t >n$, 
by Lemma \ref{Lemma sum J v} in the case where $s=1$, 
we have 
\begin{align*}
u_t v 
=\cJ_{(t-n)+n} \cdot v 
= \sum_{k=0}^{n-1} (-1)^{n-k+1} \cJ_{(t-n)+k} \cJ_1^{\lan n- k \ran} \cdot v. 
\end{align*}
By the assumption of the induction together with \eqref{cJi k ek Q=0} and \eqref{sum p e = p}, 
we have 
\begin{align*}
u_t v 
= \sum_{k=0}^{n-1} (-1)^{n-k+1} p_{t-n+k}(\g_1, \dots, \g_n) e_{n-k} (\g_1,\dots, \g_n) v
= p_t (\g_1,\dots, \g_n) v. 
\end{align*}


(\roii). 
Assume that $Q \not=0$. 
For $u_0,u_1,\dots, u_n$, 
there exist $\b, \g_1,\dots, \g_n \in \CC$ such that 
\begin{align}
\label{uk 0 to n Q NOT=0}
u_0 = n+\b \text{ and } u_k = p_k (\g_1,\dots, \g_n) +  Q^{-k} \b  \text{ for } k=1,\dots, n 
\end{align}
by Lemma \ref{Lemma solution equations}. 

By the induction on $k$, 
we prove that 
\begin{align}
\label{cJ0k ek Q NOT=0} 
\cJ_0^{\lan k \ran} \cdot v = e_k (\t_1,\t_2,\dots, \t_n) v 
\text{ for } k=1,\dots, n, 
\end{align}
where $\t_i= 1 - Q \g_i$ ($1\leq i \leq n$) 
and $e_k (\t_1,\dots, \t_n) = \sum_{1 \leq i_1 < i_2 < \dots <i_k \leq n} \t_{i_1} \t_{i_2} \dots \t_{i_k}$. 

In the case where $k=1$, 
we have $\cJ_0^{\lan 1 \ran} \cdot v = (\cJ_0 + (-Q) \cJ_1) \cdot v $. 
Then we have $\cJ_0^{\lan 1 \ran} \cdot v = e_1(\t_1,\dots, \t_n) v$ by \eqref{uk 0 to n Q NOT=0}. 

In the case where $1 < k \leq n$, by the definition, we have 
\begin{align*}
\cJ_0^{\lan k \ran} \cdot v 
&= \frac{1}{k} \sum_{z=1}^k (-1)^{z-1} \Big( \sum_{w=0}^z \begin{pmatrix} z \\ w \end{pmatrix} (-Q)^w \cJ_w \Big) 
	\cJ_{0}^{\lan k-z \ran} \cdot v. 
\end{align*}
Applying the assumption of the induction to the right-hand side, we have 
\begin{align}
\label{cJ0 k sum Q NOT=0}
\cJ_0^{\lan k \ran} \cdot v 
=  \frac{1}{k} \sum_{z=1}^k (-1)^{z-1}  e_{k-z}(\t_1,\dots, \t_n) 
	\Big( \sum_{w=0}^z \begin{pmatrix} z \\ w \end{pmatrix} (-Q)^w \cJ_w \Big)  \cdot v, 
\end{align}
where we note that $\cJ_0^{\lan 0 \ran} = e_0(\t_1,\dots, \t_n)=1$. 
On the other hand,  by \eqref{uk 0 to n Q NOT=0}, we have 
\begin{align}
\label{sum Jw p Q NOT=0}
\sum_{w=0}^z \begin{pmatrix} z \\ w \end{pmatrix} (-Q)^w \cJ_w \cdot v = p_z (\t_1,\dots, \t_n) v 
\quad (1\leq z \leq k \leq n), 
\end{align}
where we note that $\sum_{w=0}^z \begin{pmatrix} z \\ w \end{pmatrix} (-1)^w =0$. 
By \eqref{cJ0 k sum Q NOT=0} and \eqref{sum Jw p Q NOT=0} together with \eqref{e sum e p}, 
we have \eqref{cJ0k ek Q NOT=0}. 

By the induction on $t$, we prove that 
\begin{align}
\label{ut pt Q NOT=0}
u_t = p_t (\g_1,\dots, \g_n) + Q^{-t} \b \quad (t >0).
\end{align}
If $t \leq n$, \eqref{ut pt Q NOT=0} follows from \eqref{uk 0 to n Q NOT=0}. 
If $t >n$, by Lemma \ref{Lemma sum J v} in the case where $s=0$, we have 
\begin{align*}
\sum_{w=0}^n \begin{pmatrix} n \\ w \end{pmatrix} (-Q)^w \cJ_{(t-n-1) +w}^{\lan 1 \ran} \cdot v 
= \sum_{k=0}^{n-1} (-1)^{n-k+1} 
	\Big( \sum_{w=0}^k \begin{pmatrix} k \\ w \end{pmatrix} (-Q)^w \cJ_{(t-n-1) +w}^{\lan 1 \ran}\Big) 
	\cJ_0^{\lan n-k \ran} \cdot v. 
\end{align*}
By \eqref{cJ0k ek Q NOT=0} , we have 
\begin{align}
\label{sum J = e J Q NOT=0}
\begin{split}
&\sum_{w=0}^n \begin{pmatrix} n \\ w \end{pmatrix} (-Q)^w \cJ_{(t-n-1) +w}^{\lan 1 \ran} \cdot v 
\\
&= \sum_{k=0}^{n-1} (-1)^{n-k+1} e_{n-k} (\t_1,\dots, \t_n)
	\Big( \sum_{w=0}^k \begin{pmatrix} k \\ w \end{pmatrix} (-Q)^w \cJ_{(t-n-1) +w}^{\lan 1 \ran}\Big)  \cdot v. 
\end{split}
\end{align}
On the other hand, 
for $k \geq 0$, 
we have 
\begin{align}
\label{sum cJ1 t-n-1+w}
\sum_{w=0}^k \begin{pmatrix} k \\ w \end{pmatrix} (-Q)^w \cJ_{(t-n-1) +w}^{\lan 1 \ran} 
= \sum_{w=0}^{k+1} \begin{pmatrix} k+1 \\ w \end{pmatrix} (-Q)^w \cJ_{(t-n-1) +w} 
\end{align}
since $\cJ_{(t-n-1)+w}^{\lan 1 \ran} = \cJ_{(t-n-1)+w} + (-Q) \cJ_{(t-n-1) +w +1}$. 
Then, by \eqref{sum cJ1 t-n-1+w} and the assumption of the induction, we have 
\begin{align}
\label{sum 0 to n J}
\begin{split}
&\sum_{w=0}^n \begin{pmatrix} n \\ w \end{pmatrix} (-Q)^w \cJ_{(t-n-1)+w}^{\lan 1 \ran} \cdot v 
\\
&= (-Q)^{n+1} \cJ_t \cdot v 
	+ \sum_{w=0}^n \begin{pmatrix} n+1 \\ w \end{pmatrix} (-Q)^w 
		\big( p_{(t-n-1) +w} (\g_1,\dots, \g_n) + Q^{-((t-n-1)+w)} \b \big) v
\end{split} 
\end{align}
and 
\begin{align}
\label{sum 0 to k J}
\begin{split}
&\sum_{w=0}^k \begin{pmatrix} k \\ w \end{pmatrix} (-Q)^w \cJ_{(t-n-1)+w}^{\lan 1 \ran} \cdot v 
\\
&= \sum_{w=0}^{k+1} \begin{pmatrix} k+1 \\ w \end{pmatrix} (-Q)^w 
	\big( p_{(t-n-1) +w} (\g_1,\dots, \g_n) + Q^{-((t-n-1)+w)} \b \big) v 
\end{split}
\end{align}
for $k=0, 1, \dots, n-1$. 
Moreover, by the direct calculations, we have 
\begin{align}
\label{sum p b = p}
\sum_{w=0}^{k+1} \begin{pmatrix} k+1 \\ w \end{pmatrix} (-Q)^w 
	\big( p_{(t-n-1) +w} (\g_1,\dots, \g_n) + Q^{-((t-n-1)+w)} \b \big)
&= p_{k+1}^{(\Bg)} (\t_1, \dots, \t_n) 
\end{align}
for $k \geq 0$, 
where 
$p_{k+1}^{(\Bg)}(\t_1,\dots, \t_n) = \g_1^{t-n-1} \t_1^{k+1} + \g_2^{t-n-1} \t_2^{k+1} + \dots + \g_n^{t-n-1} \t_n^{k+1}$. 
Then, by \eqref{sum J = e J Q NOT=0}, \eqref{sum 0 to n J}, \eqref{sum 0 to k J} and \eqref{sum p b = p}, 
we have 
\begin{align*}
&(-Q)^{n+1} \cJ_t \cdot v - (-Q)^{n+1} \big( p_t (\g_1,\dots, \g_n) + Q^{-t} \b \big) v 
	+ p_{n+1}^{(\Bg)} (\t_1,\dots, \t_n) 
\\
&= \sum_{k=0}^{n-1} (-1)^{n-k+1} e_{n-k} (\t_1,\dots, \t_n) p_{k+1}^{(\Bg)} (\t_1,\dots, \t_n). 
\end{align*}
Applying \eqref{sum p b e = pb} to the right-hand side, we have 
\begin{align*}
\cJ_t \cdot v = \big( p_t (\g_1,\dots, \g_n) + Q^{-t} \b \big) v. 
\end{align*}
\end{proof}


\para 
By Lemma \ref{Lemma 1-dim rep rank1} and Proposition \ref{Prop ne cond rank1}, 
we see that 
the highest weight $\bu=(u_t)_{t \geq 0}$ of a  simple highest weight $U(\Fsl_2^{\lan Q \ran}[x])$-module $\cL(\bu)$ 
has the form 
\begin{align}
\label{h.w. rank1}
u_0 = \begin{cases} 
	n & \text{ if } Q=0 
	\\ n+\b &\text{ if } Q \not=0, 
	\end{cases}  
\quad 
u_t = \begin{cases} 
	p_t(\g_1,\g_2, \dots, \g_n) &\text{ if } Q=0, 
	\\ p_t (\g_1,\g_2,\dots, \g_n) + Q^{-t} \b &\text{ if } Q \not=0 
	\end{cases} 
\end{align}
for some $n \in \ZZ_{\geq 0}$ and $\b, \g_1,\g_2,\dots, \g_n \in \CC$ 
if $\cL(\bu)$ is finite dimensional.  

Let $\CC[x]$ be the polynomial ring over $\CC$ with the indeterminate variable $x$, 
and let $\CC[x]_{\mo}$ be the subset of $\CC[x]$ consisting of monic polynomials. 
We define the set $\CC[x]^{\lan Q \ran}_{\mo}$ by 
\begin{align*}
\CC[x]^{\lan Q \ran}_{\mo} = 
	\begin{cases}
	\CC[x]_{\mo} & \text{ if } Q=0, 
	\\
	\{ \vf \in \CC[x]_{\mo}  \mid Q^{-1} \text{ is not a root of } \vf\} & \text{ if } Q \not=0.
	\end{cases}
\end{align*}
Recall that 
\begin{align*}
\BB^{\lan Q \ran} = \begin{cases}
		\{0 \} & \text{ if } Q=0, 
		\\
		\CC & \text{ if } Q\not=0. 
		\end{cases} 
\end{align*}
We define the map 
\begin{align}
\label{map u}
\CC[x]^{\lan Q \ran}_{\mo} \times \BB^{\lan Q \ran} \ra \prod_{t \geq 0} \CC, 
\quad 
(\vf, \b) \mapsto \bu^{\lan Q \ran} (\vf, \b) = ( \bu^{\lan Q \ran} (\vf,\b)_t)_{t \geq 0}
\end{align}
by 
\begin{align*}
\bu^{\lan Q \ran} (\vf, \b)_t = 
	\begin{cases} 
		\deg \vf + \b & \text{ if } t=0, 
		\\
		p_t(\g_1,\g_2,\dots, \g_n) & \text{ if } t>0 \text{ and } Q =0, 
		\\
		p_t(\g_1,\g_2,\dots, \g_n) + Q^{-t} \b & \text{ if } t>0 \text{ and } Q \not=0, 
	\end{cases}
\end{align*}
when $\vf=(x-\g_1)(x-\g_2) \dots (x-\g_n)$. 
We see that the map \eqref{map u} is injective, 
and it gives a bijection between $\CC[x]^{\lan Q \ran}_{\mo} \times \BB^{\lan Q \ran}$ 
and the set of highest weight $\bu =(u_t)_{\geq 0}$ satisfying \eqref{h.w. rank1}, 
where we note that 
\begin{align*} 
p_t(\g_1, \dots, \g_n, \underbrace{Q^{-1}, \dots, Q^{-1}}_{k} \, ) + Q^{-t} \b 
= p_t (\g_1,\dots, \g_n) + Q^{-t}(\b +k). 
\end{align*} 

Then we have the following corollary of Lemma \ref{Lemma 1-dim rep rank1} and Proposition \ref{Prop ne cond rank1}. 


\begin{cor}
\label{Cor ne finite rank1}
Any finite dimensional simple $U(\Fsl_2^{\lan Q \ran}[x])$-module 
is isomorphic to $\cL(\bu^{\lan Q \ran} (\vf, \b))$ for some $(\vf, \b) \in \CC[x]^{\lan Q \ran}_{\mo} \times \BB^{\lan Q \ran}$. 
Moreover, 
$\cL(\bu^{\lan Q \ran} (\vf, \b)) \not\cong \cL(\bu^{\lan Q \ran} (\vf', \b'))$ if $(\vf, \b) \not= (\vf', \b')$. 
\end{cor}


\para
Recall the evaluation modules from the paragraph \ref{para evaluation sl2}. 
Let $L(2)$ be the two-dimensional simple $U(\Fsl_2)$-module, and $v_0 \in L(2)$ be a highest weight vector. 
We consider the evaluation module $L(2)^{\ev_\g}$ at $\g \in \CC$, 
then we see that 
\begin{align*}
\cX_t^+ \cdot v_0=0, \quad \cJ_t \cdot v_0 = \g^t v_0 
\quad (t \geq 0)
\end{align*}
in $L(2)^{\ev_{\g}}$. 

For $(\vf =(x-\g_1)(x-\g_2)\dots(x-\g_n), \b) \in \CC[x]^{\lan Q \ran}_{\mo} \times \BB^{\lan Q \ran}$, 
we consider the $U(\Fsl_2^{\lan Q \ran}[x])$-module 
\begin{align*}
\cN_{(\vf, \b)}=
L(2)^{\ev_{\g_1}} \otimes L(2)^{\ev_{\g_2}} \otimes \dots \otimes  L(2)^{\ev_{\g_n}} \otimes \cL^{\b},
\end{align*}
where $\cL^{\b}$ is the $1$-dimensional $U(\Fsl_2^{\lan Q \ran}[x])$-module given in the paragraph \ref{1-dim rep sl2}. 
Let $v_0^{(k)} \in L(2)^{\ev_{\g_k}}$ ($1 \leq k \leq n$) be a highest weight vector, and $\cL^{\b} = \CC w_0$.  
Put $v_{(\vf,\b)} = v_0^{(1)} \otimes v_0^{(2)} \otimes \dots \otimes v_0^{(n)} \otimes w_0$. 
Then, for $t \geq 0$,  we have 
\begin{align}
\label{Q1}
\cX_t^+ \cdot v_{(\vf,\b)} =0 
\end{align}
and 
\begin{align}
\label{Q2}
\cJ_t \cdot v_{(\vf,\b)}= 
	\begin{cases}
		(n+\b) v_{(\vf,\b)} & \text{ if } t=0, 
		\\
		p_t(\g_1, \g_2,\dots, \g_n) v_{(\vf,\b)}
			& \text{ if } t>0 \text{ and } Q=0, 
		\\
		(p_t(\g_1, \g_2,\dots, \g_n) + Q^{-t}\b) v_{(\vf,\b)}
			& \text{ if } t>0 \text{ and } Q\not=0. 
	\end{cases}
\end{align}
Let $\cN'_{(\vf, \b)}$ be the $U(\Fsl_2^{\lan Q \ran}[x])$-submodule of $\cN_{(\vf,\b)}$ generated by $v_{(\vf,\b)}$. 
Then \eqref{Q1} and \eqref{Q2} imply that 
$\cN'_{(\vf,\b)}$ is a highest weight module of highest weight $\bu^{\lan Q \ran}(\vf,\b)$, 
and $\cN'_{(\vf,\b)} / \rad \cN'_{(\vf,\b)}$ is isomorphic to the simple highest weight module 
$\cL (\bu^{\lan Q \ran}(\vf,\b))$. 
From the construction, 
$\cL(\bu^{\lan Q \ran} (\vf,\b)) \cong \cN'_{(\vf,\b)} / \rad \cN'_{(\vf,\b)}$ is finite dimensional 
for each $(\vf,\b) \in \CC[x]^{\lan Q \ran}_{\mo} \times \BB^{\lan Q \ran}$. 
Combining with  Corollary \ref{Cor ne finite rank1}, 
we have the following classification of finite dimensional simple $U(\Fsl_2^{\lan Q \ran}[x])$-modules. 


\begin{thm}
\label{Thm class finite simple rank1}
For $(\vf,\b) \in \CC[x]^{\lan Q \ran}_{\mo} \times \BB^{\lan Q \ran}$, 
the highest weight simple $U(\Fsl_2^{\lan Q \ran}[x])$-module $\cL(\bu^{\lan Q \ran} (\vf,\b))$ 
of highest weight $\bu^{\lan Q \ran}(\vf,\b)$ 
is finite dimensional, and we have that 
\begin{align*}
\cL(\bu^{\lan Q \ran} (\vf,\b)) \cong \cL(\bu^{\lan Q \ran} (\vf',\b')) 
\LRa 
(\vf, \b) = (\vf', \b') 
\end{align*}
for $(\vf, \b), (\vf', \b') \in \CC[x]^{\lan Q \ran}_{\mo} \times \BB^{\lan Q \ran}$. 
Moreover, 
\begin{align*}
\{ \cL(\bu^{\lan Q \ran} (\vf, \b)) \mid (\vf, \b) \in \CC[x]^{\lan Q \ran}_{\mo} \times \BB^{\lan Q \ran} \}  
\end{align*}
gives a complete set of isomorphism classes of finite dimensional simple $U(\Fsl_2^{\lan Q \ran}[x])$-modules. 
\end{thm}


\remark 
\label{Remark ev not simple}
If $Q \not=0$, the evaluation module $L(2)^{\ev_{Q^{-1}}}$ at $Q^{-1}$ is not simple. 
Recall that $L(2)$ is the two dimensional simple $U(\Fsl_2)$-module with a highest weight vector $v_0$. 
Put $v_1 = f \cdot v_0$. 
Then we see that $U(\Fsl_2^{\lan Q \ran}[x]) \cdot v_1 = \CC v_1 $ 
is a proper $U(\Fsl_2^{\lan Q \ran}[x])$-submodule of $L(2)^{\ev_{Q^{-1}}}$. 
Moreover, 
we have 
$L(2)^{\ev_{Q^{-1}}}/ \CC v_1\cong \cL^1$ 
and 
$\CC v_1 \cong \cL^{-1}$ 
as $U(\Fsl_2^{\lan Q \ran}[x])$-modules.



\section{Finite dimensional simple $U(\Fsl_m^{\lan \bQ \ran}[x])$-modules}
In this section, we classify the finite dimensional simple $U(\Fsl_m^{\lan \bQ \ran}[x])$-modules. 
By Proposition \ref{Prop simple slmx HW}, 
any finite dimensional simple $U(\Fsl_m^{\lan \bQ \ran}[x])$-module is isomorphic to 
the simple highest weight module $\cL(\bu)$ of highest weight $\bu =(u_{i,t})\in \prod_{i=1}^{m-1} \prod_{t \geq 0} \CC$. 
Thus, it is enough to classify the highest weight $\bu$ such that $\cL(\bu)$ is finite dimensional.  

\para 
{\bf $1$-dimensional representations.} 
First, we consider the $1$-dimensional representations of $\Fsl_m^{\lan Q \ran}[x]$. 
For each $i=1,2,\dots,m-1$, by checking the defining relations,  we have the homomorphism of algebras 
\begin{align} 
\label{inj sl2x to slmx}
\iota_i : U(\Fsl_2^{\lan Q_i \ran}[x]) \ra U(\Fsl_m^{\lan \bQ \ran}[x]) 
\text{ by } \cX_t^{\pm} \mapsto \cX_{i,t}^{\pm}, \,\, \cJ_t \mapsto \cJ_{i,t} \quad (t \geq 0).
\end{align}

Let $L= \CC v$ be a $1$-dimensional $U(\Fsl_m^{\lan \bQ \ran}[x])$-module. 
For each $i=1,2,\dots,m-1$, 
when we regard $L$ as a $U (\Fsl_2^{\lan Q_i \ran}[x])$-module through the homomorphism $\iota_i$, 
we see that $L$ is isomorphic to $\cL^{\b_i}$ for some $\b_i \in \BB^{\lan Q_i \ran}$ by Lemma \ref{Lemma 1-dim rep rank1}. 
Thus, we have 
\begin{align}
\label{1 dim slmx}
\cX_{i,t}^{\pm} \cdot v =0, 
\quad 
\cJ_{i,t} \cdot v = 
	\begin{cases}
		 0 & \text{ if } Q_i=0, 
		 \\
		 Q_i^{-t} \b_i v & \text{ if } Q_i \not=0 
	\end{cases} 
\quad (1 \leq i \leq m-1, \, t \geq 0)
\end{align}
for some $\Bb =(\b_i)_{1 \leq i \leq m-1} \in \prod_{i=1}^{m-1} \BB^{\lan Q_i \ran}$. 

On the other hand, by checking the defining relations, 
we can define the $1$-dimensional $U(\Fsl_m^{\lan \bQ \ran}[x])$-module 
$\cL^{\Bb} = \CC v$ by \eqref{1 dim slmx} 
for each $\Bb =(\b_i) \in \prod_{i=1}^{m-1} \BB^{\lan Q_i \ran}$. 
Now we proved the following lemma. 

\begin{lem}
\label{Lemma 1 dim slmx}
Any $1$-dimensional $U(\Fsl_m^{\lan Q \ran})$-module is isomorphic to $\cL^{\Bb}$ 
for some $\Bb \in \prod_{i=1}^{m-1} \BB^{\lan Q_i \ran}$. 
\end{lem}


\para 
For $\bu =(u_{i,t})\in \prod_{i=1}^{m-1} \prod_{t \geq 0} \CC$, 
let $v_0 $ be a highest weight vector of the simple highest weight $U(\Fsl_m^{\lan \bQ \ran}[x])$-module $\cL(\bu)$. 
When we regard $\cL(\bu)$ as a $U(\Fsl_2^{\lan Q_i \ran}[x])$-module 
through the homomorphism $\iota_i$ in \eqref{inj sl2x to slmx} 
for each $i=1,\dots, m-1$, 
we see that the $U(\Fsl_2^{\lan Q_i \ran}[x])$-submodule of $\cL(\bu)$ generated by $v_0$ 
is a highest weight $U(\Fsl_2^{\lan Q_i \ran}[x])$-module of highest weight 
$\bu_i=(u_{i,t})_{t \geq 0} \in \prod_{t \geq 0} \CC$ with the highest weight vector $v_0$. 
Then, if $\cL(\bu)$ is finite dimensional, 
we see that $\bu_i = \bu^{\lan Q_i \ran} (\vf_i, \b_i)$ for some 
$(\vf_i, \b_i) \in \CC[x]^{\lan Q_i \ran}_{\mo} \times \BB^{\lan Q_i \ran}$ 
by Theorem \ref{Thm class finite simple rank1} (or Corollary \ref{Cor ne finite rank1}). 

For $(\Bvf, \Bb) =((\vf_i, \b_i))_{1\leq i \leq m-1}  
	\in \prod_{i=1}^{m-1} \big( \CC[x]^{\lan Q_i \ran}_{\mo} \times \BB^{\lan Q_i \ran} \big)$, 
we define 
\begin{align*} 
\bu^{\lan \bQ \ran} (\Bvf, \Bb) = ( \bu^{\lan \bQ \ran}(\Bvf, \Bb)_{i,t} )_{1\leq i \leq m-1, \, t \geq 0} 
\in \prod_{i=1}^{m-1} \prod_{t \geq 0} \CC 
\end{align*} 
by 
\begin{align}
\label{def uQit}
\bu^{\lan \bQ \ran}(\Bvf, \Bb)_{i,t} 
= \begin{cases}
		\deg \vf_i + \b_i & \text{ if } t=0, 
		\\
		p_t (\g_{i,1}, \g_{i,2}, \dots, \g_{i,n_i}) & \text{ if } t >0 \text{ and } Q_i=0, 
		\\
		p_t (\g_{i,1}, \g_{i,2}, \dots, \g_{i,n_i}) + Q_i^{-t} \b_i & \text{ if } t >0 \text{ and } Q_i \not=0, 
	\end{cases}
\end{align}
when $\vf_i =( x- \g_{i,1})(x-\g_{i,2}) \dots (x-\g_{i,n_i})$ ($1\leq i \leq m-1$). 
Then we have that 
\begin{align*} 
(\bu^{\lan \bQ \ran}(\Bvf, \Bb)_{i,t})_{t \geq 0} = \bu^{\lan Q_i \ran} (\vf_i, \b_i)
\end{align*} 
for each $i=1,2,\dots, m-1$. 
From the definition, we see that 
\begin{align*}
\bu^{\lan \bQ \ran} (\Bvf, \Bb) = \bu^{\lan \bQ \ran} (\Bvf', \Bb') 
\LRa (\Bvf, \Bb) = (\Bvf', \Bb') 
\end{align*}
for $(\Bvf, \Bb),(\Bvf', \Bb') \in \prod_{i=1}^{m-1} \big( \CC[x]^{\lan Q_i \ran}_{\mo} \times \BB^{\lan Q_i \ran} \big)$. 
By the above argument, 
any finite dimensional simple $U(\Fsl_m^{\lan \bQ \ran}[x])$-module is isomorphic to 
$\cL( \bu^{\lan \bQ \ran}(\Bvf, \Bb))$ 
for some $(\Bvf, \Bb) \in \prod_{i=1}^{m-1} \big( \CC[x]^{\lan Q_i \ran}_{\mo} \times \BB^{\lan Q_i \ran} \big)$. 

On the other hand, 
for each $(\Bvf, \Bb) \in \prod_{i=1}^{m-1} \big( \CC[x]^{\lan Q_i \ran}_{\mo} \times \BB^{\lan Q_i \ran} \big)$,  
we can construct a finite dimensional highest weight $U(\Fsl_m^{\lan \bQ \ran}[x])$-module 
of highest weight $\bu^{\lan \bQ \ran}(\Bvf, \Bb)$ as follows. 

Let $\w_j$ $(1 \leq j \leq m-1)$ be the fundamental weight of $\Fsl_m$, 
and $L(\w_j)$ be the simple highest weight $U(\Fsl_m)$-module of highest weight $\w_j$. 
Let $v_0 \in L(\w_j)$ be a highest weight vector, 
then we have $e_i \cdot v_0 =0$ and $H_i \cdot v_0 = \d_{ij} v_0$ ($1\leq i \leq m-1$) by the definition. 
Recall that $L(\w_j)^{\ev_\g}$ is the evaluation module of $L(\w_j)$ at $\g \in \CC$. 
From the definition, 
we see that 
\begin{align} 
\label{action fund slm}
\cX_{i,t}^{+}  \cdot v_0 =0, 
\quad 
\cJ_{i,t} \cdot v_0 = \d_{ij} \g^t v_0 
\quad 
(1 \leq i \leq m-1, \, t \geq 0) 
\end{align}
in $L(\w_j)^{\ev_{\g}}$. 

For $(\Bvf, \Bb) =((\vf_i, \b_i))_{1\leq i \leq m-1}  
	\in \prod_{i=1}^{m-1} \big( \CC[x]^{\lan Q_i \ran}_{\mo} \times \BB^{\lan Q_i \ran} \big)$, 
we consider the $U(\Fsl_m^{\lan \bQ \ran}[x])$-module 
\begin{align*}
\cN_{(\Bvf, \Bb)} = \big( \bigotimes_{j=1}^{m-1} \bigotimes_{k=1}^{n_j} L(\w_j)^{\ev_{\g_{j, k}}} \big) \otimes \cL^{\Bb}, 
\end{align*}
where $n_j$ and $\g_{j,k}$ ($1\leq k \leq n_j$) are determined by 
$\vf_j =(x-\g_{j,1})(x- \g_{j,2}) \dots (x- \g_{j,n_j})$ for each $j=1,2,\dots, m-1$, 
and $\Bb=(\b_i)_{1 \leq i \leq m-1}$. 
Let $v_0^{(j,k)} \in L(\w_j)^{\ev_{\g_{j,k}}}$ ($1 \leq j \leq m-1$, $1 \leq k \leq n_j$) be a highest weight vector, 
and $\cL^{\Bb} = \CC w_0$. 
Put $v_{(\Bvf, \Bb)} = (\otimes_{j=1}^{m-1} \otimes_{k=1}^{n_j} v_0^{(j,k)}) \otimes w_0 \in \cN_{(\Bvf, \Bb)}$, 
then we have 
\begin{align}
\label{QQ1}
\cX_{i,t}^+ \cdot v_{(\Bvf, \Bb)} =0, 
\quad 
\cJ_{i,t} \cdot v_{(\Bvf, \Bb)} = \bu^{\lan \bQ \ran} (\Bvf, \Bb)_{i,t} \,  v_{(\Bvf, \Bb)} 
\quad (1\leq i \leq m-1, \, t \geq 0) 
\end{align}
by \eqref{action fund slm}. 
Let $\cN'_{(\Bvf, \Bb)}$ be the $U(\Fsl_m^{\lan \bQ \ran}[x])$-submodule of $\cN_{(\Bvf, \Bb)}$ 
generated by $v_{(\Bvf, \Bb)}$. 
Then \eqref{QQ1} implies that $\cN'_{(\Bvf, \Bb)}$ is a finite dimensional highest weight module of highest weight 
$\bu^{\lan \bQ \ran}(\Bvf, \Bb)$.  
Then we obtain the following classification of finite dimensional simple $U(\Fsl_m^{\lan \bQ \ran}[x])$-modules. 


\begin{thm} 
\label{Thm simple slmQ}
For $(\Bvf, \Bb) \in \prod_{i=1}^{m-1} \big( \CC[x]^{\lan Q_i \ran}_{\mo} \times \BB^{\lan Q_i \ran} \big)$,  
the highest weight simple $U(\Fsl_m^{\lan \bQ \ran}[x])$-module $\cL(\bu^{\lan \bQ \ran}(\Bvf, \Bb))$ 
of highest weight $\bu^{\lan \bQ \ran}(\Bvf, \Bb)$ is finite dimensional, 
and we have that 
\begin{align*}
\cL(\bu^{\lan \bQ \ran}(\Bvf, \Bb)) \cong \cL(\bu^{\lan \bQ \ran}(\Bvf', \Bb')) 
\LRa 
(\Bvf, \Bb) = (\Bvf', \Bb') 
\end{align*}
for $(\Bvf, \Bb), (\Bvf', \Bb') \in \prod_{i=1}^{m-1} \big( \CC[x]^{\lan Q_i \ran}_{\mo} \times \BB^{\lan Q_i \ran} \big)$
Moreover, 
\begin{align*}
\{ \cL(\bu^{\lan \bQ \ran}(\Bvf, \Bb)) \mid 
	(\Bvf, \Bb) \in \prod_{i=1}^{m-1} \big( \CC[x]^{\lan Q_i \ran}_{\mo} \times \BB^{\lan Q_i \ran} \big)
\}
\end{align*}
gives a complete set of isomorphism classes of finite dimensional simple $U(\Fsl_m^{\lan \bQ \ran}[x])$-modules. 
\end{thm}



\section{Finite dimensional simple $U(\Fgl_m^{\lan \bQ \ran}[x])$-modules}
\label{Section class simple glmQ}
In this section, 
we classify the finite dimensional simple $U(\Fgl_m^{\lan \bQ \ran}[x])$-modules. 
By Proposition \ref{Prop basis slmQ[x] and glmQ[x]} (\roiii), 
$\Fsl_m^{\lan \bQ \ran}[x]$ is a Lie subalgebra of $\Fgl_m^{\lan \bQ \ran}[x]$. 
The difference of representations of  $\Fgl_m^{\lan \bQ \ran}[x]$ 
from one of $\Fsl_m^{\lan \bQ \ran}[x]$  
is given by the family of $1$-dimensional $U(\Fgl_m^{\lan \bQ \ran}[x])$-modules 
$\{ \wt{\cL}^{\bh} \mid \bh \in \prod_{t \geq 0} \CC \}$. 
We remark that $\wt{\cL}^{\bh}$ ($\bh \in \prod_{t \geq 0} \CC$) 
is  isomorphic to the trivial representation as a $U(\Fsl_m^{\lan \bQ \ran}[x])$-module 
when we restrict the action.

\para 
{\bf $1$-dimensional representations.} 
For $\Bb = (\b_i)_{1\leq i \leq m-1} \in \prod_{i=1}^{m-1} \BB^{\lan Q_i \ran}$, 
by checking the defining relations, 
we can define the $1$-dimensional $U(\Fgl_m^{\lan \bQ \ran}[x])$-module $\wt{\cL}^{\Bb} = \CC v$ by 
\begin{align*}
&\cX_{i,t}^{\pm} \cdot v =0, 
\quad 
\cJ_{i,t} \cdot v = \begin{cases} 0 & \text{ if } Q_i=0, \\ Q_i^{-t} \b_i v & \text{ if } Q_i \not=0 \end{cases} 
\quad (1\leq i \leq m-1, \, t \geq 0), 
\\
& 
\cI_{j,t} \cdot v = \big( \sum_{k=j}^{m-1} \cJ_{k,t} \big) \cdot v 
\quad (1\leq j \leq m-1, \, t \geq 0), 
\quad 
\cI_{m,t} \cdot v =0 
\quad (t \geq 0). 
\end{align*}
Note that $\cJ_{j,t} = \cI_{j,t} - \cI_{j+1,t}$ in $U(\Fgl_m^{\lan \bQ \ran}[x])$, 
we see that $\wt{\cL}^\b \cong \cL^\b$ as $U(\Fsl_m^{\lan \bQ \ran}[x])$-modules 
when we restrict the action on $\wt{\cL}^\b$ to  $U(\Fsl_m^{\lan \bQ \ran}[x])$ 
through the injective homomorphism $\Upsilon$ in the proposition \ref{Prop basis slmQ[x] and glmQ[x]} (\roiii). 

For $\bh =(h_t)_{t \geq 0} \in \prod_{t \geq 0} \CC$, 
we can also define the $1$-dimensional $U(\Fgl_m^{\lan \bQ \ran}[x])$-module $\wt{\cL}^{\bh} = \CC v$ by 
\begin{align*}
\cX_{i,t}^{\pm } \cdot v=0, 
\quad 
\cI_{j,t} \cdot v = h_t v 
\quad (1 \leq i \leq m-1, \, 1 \leq j \leq m, \, t \geq 0).
\end{align*}
We see that $\wt{\cL}^{\bh} \cong \cL^{\mathbf{0}}$ as $U(\Fsl_m^{\lan \bQ \ran}[x])$-modules 
when we restrict the action on $\wt{\cL}^\bh$ to  $U(\Fsl_m^{\lan \bQ \ran}[x])$，
where $\mathbf{0}=(0)_{1\leq i \leq m-1} \in \prod_{i=1}^{m-1} \BB^{\lan Q _i \ran}$ 
(i.e. $\cL^{\mathbf{0}}$ is the trivial representation). 
Then we have the following classification of $1$-dimensional $U(\Fgl_m^{\lan \bQ \ran})$-modules. 


\begin{lem} 
Any $1$-dimensional $U(\Fgl_m^{\lan \bQ \ran}[x])$-module is isomorphic to 
$\wt{\cL}^{\Bb} \otimes \wt{\cL}^{\bh}$ 
for some $\Bb \in \prod_{i=1}^{m-1} \BB^{\lan Q_i \ran} $ and $\bh \in \prod_{t \geq 0} \CC$. 
We have that 
\begin{align*}
\wt{\cL}^{\Bb} \otimes \wt{\cL}^{\bh} \cong \wt{\cL}^{\Bb'} \otimes \wt{\cL}^{\bh'}
\LRa 
(\Bb, \bh) =(\Bb', \bh'). 
\end{align*}
Moreover, we see that 
$\wt{\cL}^{\Bb} \otimes \wt{\cL}^{\bh} \cong \cL^{\Bb}$ as $U(\Fsl_m^{\lan \bQ \ran}[x])$-modules 
when we restrict the action on $\wt{\cL}^{\Bb} \otimes \wt{\cL}^{\bh}$ 
to $U(\Fsl_m^{\lan \bQ \ran}[x])$. 
\end{lem}

\begin{proof} 
Let $\cL = \CC v$ be a $1$-dimensional $U(\Fgl_m^{\lan \bQ \ran})$-module. 
By restricting the action on $\cL$  to $U(\Fsl_m^{\lan \bQ \ran})$ 
through the injective homomorphism $\Upsilon$ in the proposition \ref{Prop basis slmQ[x] and glmQ[x]} (\roiii), 
we have 
\begin{align}
\label{A1}
\begin{split}
&\cX_{i,t}^{\pm} \cdot v=0, 
\\
&\cJ_{i,t} \cdot v = (\cI_{i,t} - \cI_{i+1,t}) \cdot v = 
	\begin{cases}
		0 & \text{ if } Q_i=0, 
		\\
		Q_i^{-t} \b_i v & \text{ if } Q_i \not=0 
	\end{cases}
	\quad (1 \leq i \leq m-1, \, t \geq 0) 
\end{split}
\end{align}
for some $\Bb=(\b_i) \in \prod_{i=1}^{m-1} \BB^{\lan Q_i \ran}$ by Lemma \ref{Lemma 1 dim slmx}.

On the other hand, for $t \in \ZZ_{\geq 0}$, 
there exists $h_t \in \CC$ such that 
\begin{align}
\label{A2} 
\cI_{m, t} \cdot v = h_t v 
\end{align}
since $\dim \cL =1$. 
Then \eqref{A1} and \eqref{A2} imply taht 
\begin{align*}
\cI_{j,t} \cdot v = \big(\sum_{k=j}^{m-1} \cJ_{k,t}  + h_t\big) \cdot v 
\quad (1 \leq j \leq m-1, \, t \geq 0), 
\quad 
\cI_{m,t} \cdot v = h_t v \quad (t \geq 0). 
\end{align*}
Then we see that $\cL \cong \wt{\cL}^{\Bb} \otimes \wt{\cL}^{\bh}$. 
The remaining statements are clear. 
\end{proof}


\para 
For $ \wt{\bu}=(\wt{u}_{j,t}) \in \prod_{j=1}^m \prod_{t \geq 0} \CC$, 
let $v_0$ be a highest weight vector of the simple highest weight $U(\Fgl_m^{\lan \bQ \ran}[x])$-module 
$\cL(\wt{\bu})$. 
By restricting the action on $\cL(\wt{\bu})$ to $U(\Fsl_m^{\lan \bQ \ran}[x])$, 
Theorem \ref{Thm simple slmQ} implies that 
\begin{align}
\label{Z1}
\wt{u}_{i,t} - \wt{u}_{i+1,t} = \bu^{\lan \bQ \ran}(\Bvf, \Bb)_{i,t}
\quad (1\leq i \leq m-1, \, t \geq 0)
\end{align}
for some $(\Bvf,\Bb) \in \prod_{i=1}^{m-1} (\CC[x]^{\lan Q_i \ran}_{\mo} \times \BB^{\lan Q_i \ran})$ 
if $\cL(\wt{\bu})$ is finite dimensional. 

For $t \in \ZZ_{\geq 0}$, let $h_t \in \CC$ be such that 
\begin{align}
\label{Z2}
\wt{u}_{m,t} = h_t.
\end{align}
By \eqref{Z1} and \eqref{Z2}, we have 
\begin{align*}
\wt{u}_{j,t} = \sum_{k=j}^{m-1} \bu^{\lan \bQ \ran} (\Bvf, \Bb)_{k,t} + h_t 
\quad (1\leq j \leq m-1, \, t \geq 0), 
\quad 
\wt{u}_{m,t} = h_t \quad (t \geq 0)
\end{align*}
for some $(\Bvf,\Bb) \in \prod_{i=1}^{m-1} (\CC[x]^{\lan Q_i \ran}_{\mo} \times \BB^{\lan Q_i \ran})$ 
and $\bh=(h_t) \in \prod_{t \geq 0} \CC$  
if $\cL(\wt{\bu})$ is finite dimensional. 

For $(\Bvf,\Bb, \bh) =((\vf_i, \b_i)_{1\leq i \leq m-1} , (h_t)_{t \geq 0}) 
\in \prod_{i=1}^{m-1} (\CC[x]^{\lan Q_i \ran}_{\mo} \times \BB^{\lan Q_i \ran}) \times \prod_{t \geq 0} \CC$, 
we define 
\begin{align*}
\wt{\bu}^{\lan \bQ \ran} (\Bvf, \Bb, \bh) = (\wt{\bu}^{\lan \bQ \ran} (\Bvf, \Bb, \bh)_{j,t}) 
\in \prod_{j=1}^m \prod_{t \geq 0} \CC 
\end{align*}
by 
\begin{align*}
&\wt{\bu}^{\lan \bQ \ran} (\Bvf, \Bb, \bh)_{j,t} 
	= \begin{cases} 
		\dis \sum_{k=j}^{m-1} \bu^{\lan \bQ \ran}(\Bvf, \Bb)_{k,t} + h_t 
		& \text{ if }1 \leq j \leq m-1\text{ and } t \geq 0, 
		\\
		 h_t 
		& \text{ if } j=m \text{ and } t \geq 0. 
		\end{cases}
\end{align*} 
From the definition, 
we see that 
\begin{align*}
\wt{\bu}^{\lan \bQ \ran} (\Bvf, \Bb, \bh) = \wt{\bu}^{\lan \bQ \ran} (\Bvf', \Bb', \bh') 
\LRa 
  (\Bvf, \Bb, \bh) = (\Bvf', \Bb', \bh') 
\end{align*}
for $(\Bvf, \Bb, \bh), (\Bvf', \Bb', \bh') 
\in \prod_{i=1}^{m-1} (\CC[x]^{\lan Q_i \ran}_{\mo} \times \BB^{\lan Q_i \ran}) \times \prod_{t \geq 0} \CC$. 
By the above argument, 
any finite dimensional simple $U(\Fgl_m^{\lan \bQ \ran}[x])$-module is isomorphic to 
$\cL(\wt{\bu}^{\lan \bQ \ran} (\Bvf, \Bb, \bh))$ for some 
$(\Bvf, \Bb, \bh)\in \prod_{i=1}^{m-1} (\CC[x]^{\lan Q_i \ran}_{\mo} \times \BB^{\lan Q_i \ran}) \times \prod_{t \geq 0} \CC$.


On the other hand, 
for each 
$(\Bvf, \Bb, \bh)\in \prod_{i=1}^{m-1} (\CC[x]^{\lan Q_i \ran}_{\mo} \times \BB^{\lan Q_i \ran}) \times \prod_{t \geq 0} \CC$, 
we can construct a finite dimensional highest weight $U(\Fgl_m^{\lan \bQ \ran}[x])$-module of highest weight 
$\wt{\bu}^{\lan \bQ \ran}(\Bvf, \Bb, \bh)$ as follows. 

Let $P= \bigoplus_{i=1}^m \ZZ \ve_i$ be the weight lattice of $\Fgl_m$. 
Put $\wt{\w}_l=\ve_1+\ve_2+\dots+ \ve_l$ for $l=1,2,\dots,m-1$. 
Let $L(\wt{\w}_l)$ be the simple highest weight $U(\Fgl_m)$-module of highest weight $\wt{\w}_l$, 
and $v_0 \in L(\wt{\w}_l)$ be a highest weight vector. 
Then, we have 
\begin{align*} 
e_i \cdot v_0=0 \quad (1\leq i \leq m-1) 
\text{ and } 
K_j  \cdot v_0= 
	\begin{cases}
	v_0 & \text{ if }1 \leq  j \leq l,
	\\
	0 & \text{ if }  l < j \leq m.
	\end{cases} 
\end{align*} 
Recall that $L(\wt{\w}_l)^{\wt{\ev}_{\g}}$ is the evaluation module of $L(\wt{\w}_l)$ at $\g \in \CC$. 
From the definition, we see that 
\begin{align}
\label{X1}
\cX_{i,t}^+ \cdot v_0 =0 \quad (1 \leq i \leq m-1, \, t \geq 0),
\quad 
\cI_{j,t} \cdot v_0 = \begin{cases} \g^t v_0 & \text{ if } 1 \leq j \leq l, \\ 0 & \text{ if } l < j \leq m \end{cases} 
\,\, (t \geq 0)
\end{align}
in $L(\wt{\w}_l)^{\wt{\ev}_{\g}}$. 
(We remark that $L(\wt{\w}_l)^{\wt{\ev}_\g} \cong L(\w_l)^{\ev_{\g}}$ as $U(\Fsl_m^{\lan \bQ \ran}[x])$-modules 
when we restrict the action on $L(\wt{\w}_l)^{\wt{\ev}_\g}$ to $U(\Fsl_m^{\lan \bQ \ran}[x])$.) 

For $(\Bvf,\Bb, \bh) =((\vf_i, \b_i)_{1\leq i \leq m-1} , (h_t)_{t \geq 0}) 
\in \prod_{i=1}^{m-1} (\CC[x]^{\lan Q_i \ran}_{\mo} \times \BB^{\lan Q_i \ran}) \times \prod_{t \geq 0} \CC$, 
we consider the $U(\Fgl_m^{\lan \bQ \ran}[x])$-module 
\begin{align*}
\wt{\cN}_{(\Bvf, \Bb, \bh)} = \big( \bigotimes_{l=1}^{m-1} \bigotimes_{k=1}^{n_l} 
	L(\wt{\w}_l)^{\wt{\ev}_{\g_{l, k}}} \big) \otimes \wt{\cL}^{\Bb} \otimes \wt{\cL}^{\bh}, 
\end{align*}
where $n_l$ and $\g_{l,k}$ ($1\leq k \leq n_l$) are determined by 
$\vf_l =(x-\g_{l,1})(x- \g_{l,2}) \dots (x- \g_{l,n_l})$ for each $l=1,2,\dots, m-1$, 
and we put $\Bb=(\b_i)_{1 \leq i \leq m-1}$ and $\bh=(h_t)_{t \geq 0}$. 
Let $v_0^{(l,k)} \in L(\wt{\w}_l)^{\wt{\ev}_{\g_{l,k}}}$ ($1 \leq l \leq m-1$, $1 \leq k \leq n_l$) be a highest weight vector, 
$\wt{\cL}^{\Bb} = \CC w_0$ and $ \wt{\cL}^{\bh} = \CC z_0$. 
Put $v_{(\Bvf, \Bb, \bh)} = (\otimes_{l=1}^{m-1} \otimes_{k=1}^{n_l} v_0^{(l,k)}) \otimes w_0 \otimes z_0  
\in \wt{\cN}_{(\Bvf, \Bb)}$, 
then we have 
\begin{align}
\label{X2}
\cX_{i,t}^+ \cdot v_{(\Bvf, \Bb, \bh)} =0, 
\quad 
\cI_{j,t} \cdot v_{(\Bvf, \Bb, \bh)} = \wt{\bu}^{\lan \bQ \ran}(\Bvf, \Bb, \bh)_{j,t} \,  v_{(\Bvf, \Bb, \bh)} 
\end{align}
for $1 \leq i \leq m-1$, $1\leq j \leq m$ and  $ t \geq 0$ by \eqref{X1}. 
Let $\wt{\cN}'_{(\Bvf, \Bb, \bh)}$ be the $U(\Fgl_m^{\lan \bQ \ran}[x])$-submodule of $\wt{\cN}_{(\Bvf, \Bb, \bh)}$ 
generated by $v_{(\Bvf, \Bb, \bh)}$. 
Then \eqref{X2} implies that $\wt{\cN}'_{(\Bvf, \Bb, \bh)}$ is a finite dimensional highest weight module of highest weight 
$\wt{\bu}^{\lan \bQ \ran}(\Bvf, \Bb, \bh)$.  
Now we obtain the following classification of finite dimensional simple $U(\Fgl_m^{\lan \bQ \ran}[x])$-modules. 


\begin{thm} 
\label{Thm class simple glmQ}
For $(\Bvf, \Bb, \bh) 
\in \prod_{i=1}^{m-1} \big( \CC[x]^{\lan Q_i \ran}_{\mo} \times \BB^{\lan Q_i \ran} \big) \times \prod_{t \geq 0} \CC$,  
the highest weight simple $U(\Fgl_m^{\lan \bQ \ran}[x])$-module $\cL(\wt{\bu}^{\lan \bQ \ran}(\Bvf, \Bb, \bh))$ 
of highest weight $\wt{\bu}^{\lan \bQ \ran}(\Bvf, \Bb, \bh)$ is finite dimensional, 
and we have that 
\begin{align*}
\cL(\wt{\bu}^{\lan \bQ \ran}(\Bvf, \Bb, \bh)) \cong \cL(\wt{\bu}^{\lan \bQ \ran}(\Bvf', \Bb', \bh')) 
\LRa 
(\Bvf, \Bb, \bh) = (\Bvf', \Bb', \bh') 
\end{align*}
for $(\Bvf, \Bb, \bh), (\Bvf', \Bb', \bh') 
\in \prod_{i=1}^{m-1} \big( \CC[x]^{\lan Q_i \ran}_{\mo} \times \BB^{\lan Q_i \ran} \big) \times \prod_{t \geq 0} \CC$. 
Moreover, 
\begin{align*}
\{ \cL(\wt{\bu}^{\lan \bQ \ran}(\Bvf, \Bb, \bh)) \mid 
	(\Bvf, \Bb, \bh) \in \prod_{i=1}^{m-1} 
		\big( \CC[x]^{\lan Q_i \ran}_{\mo} \times \BB^{\lan Q_i \ran} \big) \times \prod_{t \geq 0} \CC 
\}
\end{align*}
gives a complete set of isomorphism classes of finite dimensional simple $U(\Fgl_m^{\lan \bQ \ran}[x])$-modules. 
\end{thm}

We also have the following corollary. 

\begin{cor}
For $(\Bvf, \Bb, \bh) 
\in \prod_{i=1}^{m-1} \big( \CC[x]^{\lan Q_i \ran}_{\mo} \times \BB^{\lan Q_i \ran} \big) \times \prod_{t \geq 0} \CC$,  
we have 
\begin{align*}
\cL(\wt{\bu}^{\lan \bQ \ran}(\Bvf, \Bb, \bh)) \cong \cL( \bu^{\lan \bQ \ran} (\Bvf, \Bb)) 
\text{ as $U(\Fsl_m^{\lan \bQ \ran}[x])$-modules} 
\end{align*}
when we restrict the action on $\cL(\wt{\bu}^{\lan \bQ \ran}(\Bvf, \Bb, \bh))$ to
 $U(\Fsl_m^{\lan \bQ \ran}[x])$. 
\end{cor}

\begin{proof} 
We prove that 
$\cL(\wt{\bu}^{\lan \bQ \ran}(\Bvf, \Bb, \bh))$ is also simple 
when we restrict the action to $U(\Fsl_m^{\lan \bQ \ran}[x])$. 
Then the isomorphism follows from the definitions of 
$\wt{\bu}^{\lan \bQ \ran}(\Bvf, \Bb, \bh)$ and $\bu^{\lan \bQ \ran} (\Bvf, \Bb)$. 

Let $v_0 \in \cL(\wt{\bu}^{\lan \bQ \ran}(\Bvf, \Bb, \bh))$ 
be a highest weight vector as the $U(\Fgl_m^{\lan \bQ \ran}[x])$-module. 
Then we have 
\begin{align*} 
\cL(\wt{\bu}^{\lan \bQ \ran}(\Bvf, \Bb, \bh)) = U(\fn^-) \cdot v_0 
\end{align*} 
by the triangular decomposition in Proposition \ref{Prop basis slmQ[x] and glmQ[x]} (\roiv). 
This implies that 
\begin{align}
\label{XX1} 
\cL(\wt{\bu}^{\lan \bQ \ran}(\Bvf, \Bb, \bh)) = U(\Fsl_m^{\lan \bQ \ran}[x]) \cdot v_0. 
\end{align}

Assume that $\cL(\wt{\bu}^{\lan \bQ \ran}(\Bvf, \Bb, \bh))$ is not simple as a $U(\Fsl_m^{\lan \bQ \ran}[x])$-module 
by the restriction, 
then $\cL(\wt{\bu}^{\lan \bQ \ran}(\Bvf, \Bb, \bh))$ contains a non-zero proper simple $U(\Fsl_m^{\lan \bQ \ran}[x])$-submodule 
which is a highest weight $U(\Fsl_m^{\lan \bQ \ran}[x])$-module. 
This implies  that there exist an element $w_0 \in \cL(\wt{\bu}^{\lan \bQ \ran}(\Bvf, \Bb, \bh))$ 
such that $\cX_{i,t}^+ \cdot w_0  =0$ ($1\leq i \leq m-1, \, t \geq 0)$ 
and $w_0 \not \in \CC v_0$. 
Then $U(\Fgl_m^{\lan \bQ \ran}[x]) \cdot w_0$ turns out to be a non-zero proper $U(\Fgl_m^{\lan \bQ \ran}[x])$-submodule 
of $\cL(\wt{\bu}^{\lan \bQ \ran}(\Bvf, \Bb, \bh))$. 
This is a contradiction. 
\end{proof}


\appendix 

\section{Some combinatorics}
\para 
Let $\ZZ[x_1,\dots, x_n]$  be the ring of polynomials in  independent variables $x_1, \dots, x_n$ over $\ZZ$. 
For $k \in \ZZ_{>0}$, put   
\begin{align*}
& p_k(x_1,\dots, x_n) = x_1^k+ x_2^k + \dots + x_n^k \in \ZZ [x_1,\dots, x_n],
\\
& e_k (x_1,\dots, x_n) = \sum_{1 \leq i_1 < i_2< \dots < i_k \leq n} x_{i_1} x_{i_2} \dots x_{i_k} 
	\in \ZZ[x_1,\dots, x_n]. 
\end{align*} 
Namely, 
$p_k(x_1,\dots,x_n)$ is the power sum symmetric polynomial of degree $k$, 
and $e_k(x_1,\dots, x_n)$ is the elementary symmetric polynomial of degree $k$. 
We also put $e_0 (x_1,\dots, x_n)=1$. 
Then, for $k >0$, we have 
\begin{align}
\label{e sum e p}
k e_k(x_1,\dots, x_n) = \sum_{z=1}^k (-1)^{z-1} p_z (x_1,\dots, x_n) e_{k-z} (x_1,\dots, x_n) 
\end{align}
by \cite[\S1 (2.11')]{Mac}.
For $s >n$, we have 
\begin{align*}
0
&= \sum_{z=1}^s (-1)^{z-1} p_z (x_1,\dots, x_n) e_{s-z} (x_1,\dots, x_n)
\\
&= \sum_{z=1}^{s-1} (-1)^{z-1} p_z (x_1,\dots, x_n) e_{s-z} (x_1,\dots, x_n) 
	+ (-1)^{s-1} p_s (x_1,\dots, x_n) 
\\
&= \sum_{z=s-n}^{s-1} (-1)^{z-1} p_z (x_1,\dots, x_n) e_{s-z} (x_1,\dots, x_n) 
	+ (-1)^{s-1} p_s (x_1,\dots, x_n), 
\end{align*}
where we note that $e_{s-z} (x_1,\dots, x_n)=0$ if $z < s-n$. 
Put $w = z-s +n$, we have 
\begin{align}
\label{sum p e = p}
\sum_{w=0}^{n-1} (-1)^{n - w +1} p_{s-n +w} (x_1,\dots, x_n) e_{n-w} (x_1,\dots, x_n) 
	=  p_s (x_1,\dots, x_n)
\end{align}
for $s >n$. 


\begin{lem}
\label{Lemma solution equations}
For $n \in \ZZ_{>0}$ and $u_1, u_2,\dots, u_n \in \CC$, 
the simultaneous equations 
\begin{align}
\label{equations 1}
\begin{cases}
p_1 (x_1, x_2,\dots, x_n) = u_1, 
\\
p_2 (x_1, x_2,\dots, x_n)=u_2, 
\\
\quad \vdots 
\\
p_n (x_1,x_2,\dots, x_n)=u_n 
\end{cases}
\end{align}
has a solution in $\CC$. 
\end{lem}

\begin{proof} 
We prove the lemma by the induction on $n$. 
In the case where $n=1$, it is clear. 
If $n >1$, 
the equations \eqref{equations 1} are equivalent to the equations 
\begin{align}
\label{equations 2}
\begin{cases} 
p_1(x_1, x_2,\dots, x_{n-1}) = u_1 - x_n, 
\\
p_2 (x_1,x_2,\dots, x_{n-1}) = u_2 - x_n^2, 
\\
\quad \vdots 
\\
p_{n-1}(x_1, x_2, \dots, x_{n-1}) = u_{n-1} - x_n^{n-1}, 
\\
p_n (x_1,x_2,\dots, x_{n-1}) = u_n - x_n^{n}. 
\end{cases}
\end{align}

By \eqref{e sum e p}, 
we have 
\begin{align*}
&p_n (x_1,x_2,\dots, x_{n-1}) 
\\
&= \sum_{i=1}^{n-1} (-1)^{i+n-1} p_i (x_1,x_2,\dots, x_{n-1}) e_{n-i} (x_1,x_2,\dots, x_{n-1}), 
\end{align*}
where we note that $e_n (x_1,x_2,\dots, x_{n-1})=0$. 
On the other hand, 
we can write 
\begin{align*}
e_{n-i} (x_1,x_2,\dots, x_{n-1}) 
= \sum_{\la \vdash n-i} \a_{\la} p_{\la} (x_1,x_2,\dots, x_{n-1}) 
\end{align*}
for some $\a_{\la} \in \CC$, 
where 
$p_{\la}(x_1,x_2,\dots, x_{n-1}) = \prod_{j = 1}^{\ell(\la)} p_{\la_j}(x_1,x_2,\dots, x_{n-1})$ 
for $\la=(\la_1, \la_2, \dots ) \vdash n-i$. 
Thus we have 
\begin{align*}
p_n (x_1,x_2,\dots, x_{n-1}) 
= \sum_{i=1}^{n-1} \sum_{\la \vdash n-i} (-1)^{i+n-1} \a_{\la} p_i (x_1,x_2, \dots, x_{n-1}) p_{\la}(x_1,x_2,\dots, x_{n-1}).  
\end{align*}
(Note that $\{p_{\mu} (x_1,x_2,\dots, x_{n-1}) \mid \mu \vdash k\}$ is not linearly independent if $k \geq n$. 
For an example, we have $p_{(3)}(x_1,x_2)= \frac{3}{2} p_{(2,1)}(x_1,x_2) - \frac{1}{2} p_{(1,1,1)}(x_1,x_2)$.) 

Then the equations \eqref{equations 2} are equivalent to the equations 
\begin{align}
\label{equations 3}
\begin{cases} 
p_1(x_1, x_2,\dots, x_{n-1}) = u_1 - x_n, 
\\
p_2 (x_1,x_2,\dots, x_{n-1}) = u_2 - x_n^2, 
\\
\quad \vdots 
\\
p_{n-1}(x_1, x_2, \dots, x_{n-1}) = u_{n-1} - x_n^{n-1}, 
\\
\dis 
\sum_{i=1}^{n-1} \sum_{\la \vdash n-i} (-1)^{i+n-1} \a_{\la}  
	(u_i - x_n^i) \prod_{j = 1}^{\ell (\la)} (u_{\la_j} - x_n^{\la_j})
	= u_n - x_n^{n} 
	\quad \cdots (\ast 1). 
\end{cases}
\end{align}

Let $\b_n$ be a solution of the equation $(\ast 1)$ for the variable $x_n$. 
By the assumption of the induction, 
the simultaneous equations 
\begin{align*}
\begin{cases} 
p_1 (x_1,x_2, \dots, x_{n-1}) = u_1 - \b_n, 
\\
p_2 (x_1,x_2,\dots, x_{n-1}) =u_2 - \b_n^2, 
\\
\quad \vdots 
\\
p_{n-1}(x_1,x_2,\dots, x_{n-1}) = u_{n-1} - \b_n^{n-1} 
\end{cases}
\end{align*}
for variables $x_1, x_2,\dots, x_{n-1}$ has a solution. 
We denote it by $(x_1,x_2,\dots, x_{n-1})=(\b_1, \b_2, \dots, \b_{n-1})$. 
Then $(x_1,x_2,\dots, x_n)=(\b_1,\b_2, \dots, \b_n)$ gives a solution of \eqref{equations 3}. 
\end{proof}

\para 
We consider some  modifications of the formulas  \eqref{e sum e p} and \eqref{sum p e = p} as follows. 
Let $\bb =(b_1, \dots, b_n)$ be $n$ independent variables, 
and we consider the ring of polynomials 
$\ZZ[x_1,\dots, x_n][b_1,\dots, b_n]$. 
For $k \in \ZZ_{>0}$, put 
\begin{align*}
& e_k^{(\bb)} ( x_1, \dots,  x_n) 
\\
& = \sum_{1 \leq i_1 < i_2 < \dots < i_k \leq n} 
	(b_{i_1} + b_{i_2} +  \dots + b_{i_k}) x_{i_1} x_{i_2} \dots x_{i_k} 
	\in \ZZ[x_1,\dots, x_n][b_1,\dots, b_n]
\end{align*} 
and 
\begin{align*}
& p_k^{(\bb)}(x_1,\dots, x_n) = b_1 x_1^k + b_2 x_2^{k} + \dots + b_n x_n^{k} 
	\in \ZZ[x_1,\dots, x_n][b_1,\dots, b_n]. 
\end{align*}
We also put $e_0^{(\bb)}=1$. 
Note that $e_k^{(\bb)}(x_1,\dots,x_n)=0$ if $k >n$. 
Put $\mathbf{1}=(1,1,\dots,1)$, 
then we have 
$e_k^{(\mathbf{1})}(x_1,\dots, x_n) = k e_k (x_1,\dots, x_n)$ 
and 
$p_k^{(\mathbf{1})}(x_1,\dots, x_n) = p_k(x_1,\dots, x_n)$. 

We consider the generating functions $E(t)$, $E^{(\bb)}(t)$ and $P^{(\bb)}(t)$ by 
\begin{align*}
& E(t) = \sum_{k \geq 0} e_k(x_1,\dots, x_n) t^k 
	\in  \ZZ[x_1,\dots, x_n][b_1,\dots, b_n] [[t]],
\\
& E^{(\bb)} (t) = \sum_{k \geq 0} e_{k+1}^{(\bb)}(x_1,\dots, x_n) t^k 
	\in  \ZZ[x_1,\dots, x_n][b_1,\dots, b_n] [[t]], 
\\
& P^{(\bb)}(t) = \sum_{k \geq 0} (-1)^k p_{k+1}^{(\bb)} (x_1,\dots, x_n) t^k 
	\in  \ZZ[x_1,\dots, x_n][b_1,\dots, b_n] [[t]]. 
\end{align*}
Then, we have 
\begin{align*}
E(t) = \prod_{i=1}^n (1 + x_i t), 
\quad 
P^{(\bb)}(t) = \sum_{i=1}^n \frac{b_i x_i}{1 + x_i t} 
\end{align*}
and 
\begin{align*}
P^{(\bb)} (t) E(t) 
= 
\sum_{i=1}^n b_i x_i \Big( \prod_{j =1, j\not=i}^n (1 + x_i t) \Big)
= E^{(\bb)} (t). 
\end{align*}
This implies that, for $k \geq 0$, 
\begin{align}
e_{k+1}^{(\bb)} (x_1,\dots, x_n) 
= \sum_{z=0}^k (-1)^z p_{z+1}^{(\bb)} (x_1,\dots, x_n) e_{k-z} (x_1,\dots, x_n). 
\end{align} 
In the case where $k=n$, 
we have 
\begin{align*}
\sum_{z=0}^n (-1)^z p_{z+1}^{(\bb)}  (x_1, \dots, x_n) e_{n-z} (x_1,\dots, x_n) 
= 0
\end{align*}
since $e_{n+1}^{(\bb)}(x_1, \dots, x_n)=0$. 
This implies that 
\begin{align}
\label{sum p b e = pb}
\sum_{z=0}^{n-1} (-1)^{n-z +1} p_{z+1}^{(\bb)} (x_1,\dots, x_n) e_{n-z} (x_1,\dots, x_n) 
=  p_{n+1}^{(\bb)} (x_1,\dots, x_n). 
\end{align}



\end{document}